\documentclass[final,leqno,onefignum,onetabnum]{siamltex1213}

\usepackage[english]{babel}
\usepackage{amsmath}
\usepackage{amssymb}
\usepackage{enumerate}
\usepackage{enumitem}
\usepackage{graphicx}
\usepackage{array} 
\usepackage{hyperref}
\usepackage{color}
\usepackage{nicefrac}
\usepackage[it,IT]{subfigure}

\usepackage{vmargin}
\setpapersize{USletter}
\setmarginsrb{2.5cm}{2.0cm}{2.5cm}{2.0cm}{0.8cm}{0.4cm}{0.4cm}{0.2cm}

\DeclareGraphicsExtensions{.eps,.pdf,.png,.jpg}

\newcommand{\R}{\ensuremath{\mathbb{R}}}

\newcommand{\Z}{\ensuremath{\mathbb{Z}}}
\newcommand{\Hil}{\ensuremath{\mathcal{H}}}

\newcommand{\V}{\ensuremath{\mathcal{V}}}
\newcommand{\A}{\ensuremath{\mathcal{A}}}
\large

\newcommand{\eps}{\ensuremath{\varepsilon}}
\renewcommand{\phi}{\ensuremath{\varphi}}
\providecommand{\abs}[1]{\lvert#1\rvert}
\providecommand{\norm}[1]{\lVert#1\rVert}
\providecommand{\Norm}[1]{\left\lVert#1\right\rVert}

\newcommand{\HE}{E}    
\newcommand{\Henry}{H} 
\newcommand{\degK}{{}^\circ\hspace*{-0.1em}K}
\newcommand{\degC}{{}^\circ\hspace*{-0.1em}C}
\newcommand{\tend}{{t_m}}
\newcommand{\Yone}{Y^{1}}
\newcommand{\Ytwo}{Y^{2}}

\newcommand{\Omeps}{\Omega_{\eps}}
\newcommand{\Gameps}{\Gamma_{\eps}}
\newcommand{\Taeps}{T_{\alpha,\eps}}
\newcommand{\Toneeps}{T_{1,\eps}}
\newcommand{\Ttwoeps}{T_{2,\eps}}
\newcommand{\HEaeps}{\HE_{\alpha,\eps}}
\newcommand{\HEoneeps}{\HE_{1,\eps}}
\newcommand{\HEtwoeps}{\HE_{2,\eps}}
\newcommand{\Omoneeps}{\Omega_{\eps}^{1}}
\newcommand{\Omtwoeps}{\Omega_{\eps}^{2}}
\newcommand{\half}{\frac{1}{2}}
\newcommand{\bigC}{C}
\newcommand{\Tout}{T_{a}}
\newcommand{\Tcrit}{T_{c}}
\newcommand{\defeq}{:=}
\newcommand{\leavethisout}[1]{}
\newcommand{\normvec}{\mathbf{n}}
\newcommand{\ds}{\displaystyle}
\newcommand{\gasR}{{\cal R}}
\newcommand{\Mmicro}{m} 
\newcommand{\Mmacro}{M} 
\newcommand{\Rtree}{R_{\text{\it tree}}}
\newcommand{\Rf}{R^f}
\newcommand{\Lf}{L^f}
\newcommand{\Vf}{V^f}

\newcommand{\Lv}{L^v}
\newcommand{\Vv}{V^v}
\newcommand{\eqdef}{\mathrel{\mathop=}:}

\newcommand{\reptheorem}[2]{%
  {\renewcommand{\thetheorem}{\ref{#1}}
    \begin{theorem}
      #2
    \end{theorem}
    \addtocounter{theorem}{-1}
  }} 

\newcommand{\replemma}[2]{%
  {\renewcommand{\thetheorem}{\ref{#1}}
    \begin{lemma}
      #2
    \end{lemma}
    \addtocounter{theorem}{-1}
  }} 

\newenvironment{propShowalter}{%
  \par{\scshape Showalter's Proposition (Prop.~4.1 in~\cite{PDE5})}\itshape}{}
 
\newcounter{question}

\title{
 Homogenization of the Stefan problem, \\
  with application to maple sap exudation\footnotemark[1]}

\author{Isabell Graf\footnotemark[2]\ \footnotemark[3]
  \and John M. Stockie\footnotemark[3]}

\begin{document}

\maketitle
\slugger{mms}{xxxx}{xx}{x}{x--x}

\renewcommand{\thefootnote}{\fnsymbol{footnote}} 

\footnotetext[0]{Latest revision: \today.}

\footnotetext[1]{This work was supported in part by a Discovery Grant
  from the Natural Sciences and Engineering Research Council of Canada.}

\footnotetext[2]{Corresponding author.  IG was supported by a Feodor
  Lynen Fellowship from the Alexander von Humboldt Foundation.}

\footnotetext[3]{Department of Mathematics, Simon Fraser University,
  8888 University Drive, Burnaby, BC, V5A 1S6, Canada
  (\email{grafisab@gmail.com}, \email{stockie@math.sfu.ca}).}

\renewcommand{\thefootnote}{\arabic{footnote}}

\begin{abstract}
  The technique of periodic homogenization with two-scale convergence is
  applied to the analysis of a two-phase Stefan-type problem that arises
  in the study of a periodic array of melting ice bars.  For this
  ``reduced model'' we prove results on existence, uniqueness and
  convergence of the two-scale limit solution in the weak form, which
  requires solving a macroscale problem for the global temperature field
  and a reference cell problem at each point in space which captures the
  underlying phase change process occurring on the microscale.  We state
  a corresponding strong formulation of the limit problem and use it to
  design an efficient numerical solution algorithm.  The same
  homogenized temperature equations are then applied to solve a much
  more complicated problem involving multi-phase flow and heat transport
  in trees, where the sap is present in both frozen and liquid forms and
  a third gas phase is also present.  Our homogenization approach has
  the advantage that the global temperature field is a solution of the
  same reduced model equations, while all the remaining physics are
  relegated to the reference cell problem.  Numerical simulations are
  performed to validate our results and draw conclusions regarding the
  phenomenon known as sap exudation, which is of great importance in
  sugar maple trees and few other related species.
\end{abstract}



\pagestyle{myheadings}
\thispagestyle{plain}
\markboth{I. GRAF AND J. M. STOCKIE}{HOMOGENIZATION OF THE STEFAN PROBLEM}

\section{Introduction}
\label{sec:intro}

Multiscale problems are characterized by geometric, material or other
features that exhibit variations on spatial or temporal scales that
differ widely.  Many mathematical and numerical methods have been
derived for dealing with this scale separation and accurately capturing
the interactions between phenomena operating on different scales.  For
problems that have a periodic microstructure, one such
mathematical approach is periodic homogenization, and more specifically
the approach known as two-scale asymptotic convergence~\cite{TWOSCALE1}.

In this paper, we are interested in applying two-scale convergence to
analyze solutions of a Stefan-type problem that arises from multi-phase
flow in a porous medium coupled with heat transport.  In particular, we
are motivated by the study of sap flow in a maple tree that is subject
to both thawing and freezing~\cite{AHORN5}.  We seek insight into the
phenomenon called \emph{sap exudation}, which refers to the generation
of elevated sap pressure within the maple tree trunk in advance of the
spring thaw when the tree is still in a leafless state and no
photosynthesis (or transpiration) occurs to drive the sap flow.  Our
work is based on the model derived in~\cite{AHORN5} that captures the
physical processes going on at the microscale (that is, at the level of
individual wood cells) and includes multiphase flow of gas/ice/water,
heat transport, porous flow through cell walls, and osmotic effects.
There is an inherent repeating structure in wood at the cellular scale
that lends itself naturally to a homogenization approach, although with
the exception of the analysis of \cite{chavarriakrauser-ptashnyk-2013}
for water transport in plant tissues, the application of homogenization
techniques in the context of heat or sap flow in trees is very limited.

Before tackling the sap exudation problem in its full complexity, it is
convenient to first develop our homogenization approach in the context
of a simpler ``reduced model'' that focuses only on heat transport and
ice/water phase change for a similarly fine-structured domain.  To this
end, we consider in Section~\ref{sec:problem1} the transport of heat in
a region containing a periodic array of small cylindrical ice bars
immersed in water. The temperature field obeys the heat equation, and is
coupled with a Stefan condition at the ice/water phase interface that
governs the transition from solid to liquid and vice versa.  Many
different approaches have been developed to analyze such phase
transitions, and these are well-described in the monograph
\cite{STEFAN5}.  To handle the multiplicity of the ice bars, we apply
the technique of periodic homogenization with two-scale convergence
established in \cite{TWOSCALE1, TWOSCALE2}. Several authors have
previously applied homogenization to solving the Stefan problem, such as
in \cite{bossavit-damlamian-1981, STEFAN1, STEFAN2} where the phase
change boundary is handled by separately homogenizing an auxiliary
problem. In \cite{STEFAN3} on the other hand, an additional function
$\theta$ is introduced for an aggregate state that diffuses on a slow
time scale and with which all microscopic phase changes are properly
captured.

The approach we employ in this paper has the advantage that it applies
homogenization techniques in a straightforward manner in order to obtain
an uncomplicated limit model, the simplicity of which ensures that
numerical simulations are relatively easy to perform.  More
specifically, we define a reference cell $Y$ that is divided into two
sub-regions: $\Yone$ where the temperature diffuses rapidly; and $\Ytwo$
where we define a second temperature field that diffuses slowly. One
particular challenge that arises in the study of Stefan problems is that
the diffusion coefficient depends on the underlying phases, so that heat
diffuses differently in water or ice.  Consequently, the diffusion
coefficient depends on temperature (or equivalently on enthalpy) so that
the governing differential equation is only quasi-linear.

The details of our periodic homogenization approach are introduced in
the context of the reduced model in Section~\ref{sec:description}.  We
prove the existence of a solution of the governing PDE in
Section~\ref{sec:existence} and derive {a~priori} estimates of the
solution in Section~\ref{sec:a-priori}.  The identification of the
two-scale limit equations is performed in Section~\ref{sec:two-scale},
and uniqueness of the limit problem is proven in
Section~\ref{sec:unique}.  Finally, we provide the strong formulation
and one-phase formulation of the limit problem in
Section~\ref{sec:strong_form}.  Detailed proofs of the lemmas and
theorems involved are relegated to the Appendices.  We conclude our
study of the reduced model in Section~\ref{sec:sims-reduced} with
numerical simulations that illustrate the behaviour of the homogenized
solution.

The remainder of the paper focuses on the application of our periodic
homogenization approach to the maple sap exudation problem.  We aim to
extend the homogenized equation derived earlier for the reduced model to
obtain a corresponding homogenized equation that governs
freezing/thawing in maple sap, and this process is carefully justified
in the introduction to Section~\ref{sec:maple_sap_exudation}.  The
governing equations for the sap exudation problem on the cellular scale
are a slightly modified version of the model developed in~\cite{AHORN5}
which we summarize in Section~\ref{sec:maple_description}.  The major
advantage of our approach is that the homogenized equations for the
two-phase reduced problem can also be used to determine the temperature
in the more complicated sap exudation problem, with the added complexity
of the extra gas phase, dynamic phase interfaces, pressure exchange,
etc.\ all confined to the reference cell.  Guided by our experience with
the reduced model, we develop a corresponding numerical algorithm for
solving the homogenized sap exudation model, and then present
simulations in Section~\ref{sec:maple_simul} that we compare with
results from the reduced model.  Finally, we conclude in
Section~\ref{sec:discussion} with a discussion of the biological
significance of our results in the context of tree sap flow, and we
suggest several exciting avenues for future work that exploit our
homogenized sap exudation model.

\section{Reduced model: Ice bars melting in water}
\label{sec:problem1}

\subsection{Mathematical formulation}
\label{sec:description}

Let $\Omega\subset \R^n$ be a bounded domain having Lipschitz boundary
and that contains both water and ice, where the ice takes the form of
``bars'' occurring in a regular repeating pattern such as that shown in
Figure~\ref{fig:ref-cell-a}.  Let $Y = [0,\delta]^n$ be a
\emph{reference cell} that captures the configuration of the periodic
microstructure, and for which $\delta$ represents its physical size and
$0\leq\delta\ll 1$.  The reference cell is divided into two sub-domains
$\Yone$ and $\Ytwo$ that are separated by a Lipschitz boundary
$\Gamma=\Yone\cap \Ytwo$.  The primary feature that we exploit in our
homogenization approach is that within $\Yone$ heat diffuses rapidly,
whereas in $\Ytwo$ there is instead a relatively slow diffusion of heat.
We take $\Gamma$ to be a circle of radius $\gamma$ that satisfies
$0<\gamma<\delta$.  A picture illustrating this decomposition of the
reference cell is given in Figure~\ref{fig:ref-cell-b}.

We next introduce a small parameter $0<\eps\ll 1$ that corresponds to
the size of the periodic microstructure.  The domain $\Omega$ may then
be decomposed into three $\eps$-dependent sub-domains: $\Omoneeps\defeq
\bigcup_{k\in\Z^n}\eps(k+\Yone)\cap \Omega$ (which is connected), and
two disconnected components consisting of the region $\Omtwoeps \defeq
\bigcup_{k\in\Z^n}\eps(k+\Ytwo)\cap\Omega$ and the boundary curves
$\Gameps \defeq \bigcup_{k\in\Z^n} \eps(k+\Gamma)\cap\Omega$.  This
decomposition is highlighted in Figure~\ref{fig:ref-cell-a}.
\begin{figure}
  \centering
  \subfigure[Periodically-tiled domain
  $\Omega$.]{\includegraphics[width=0.5\textwidth]{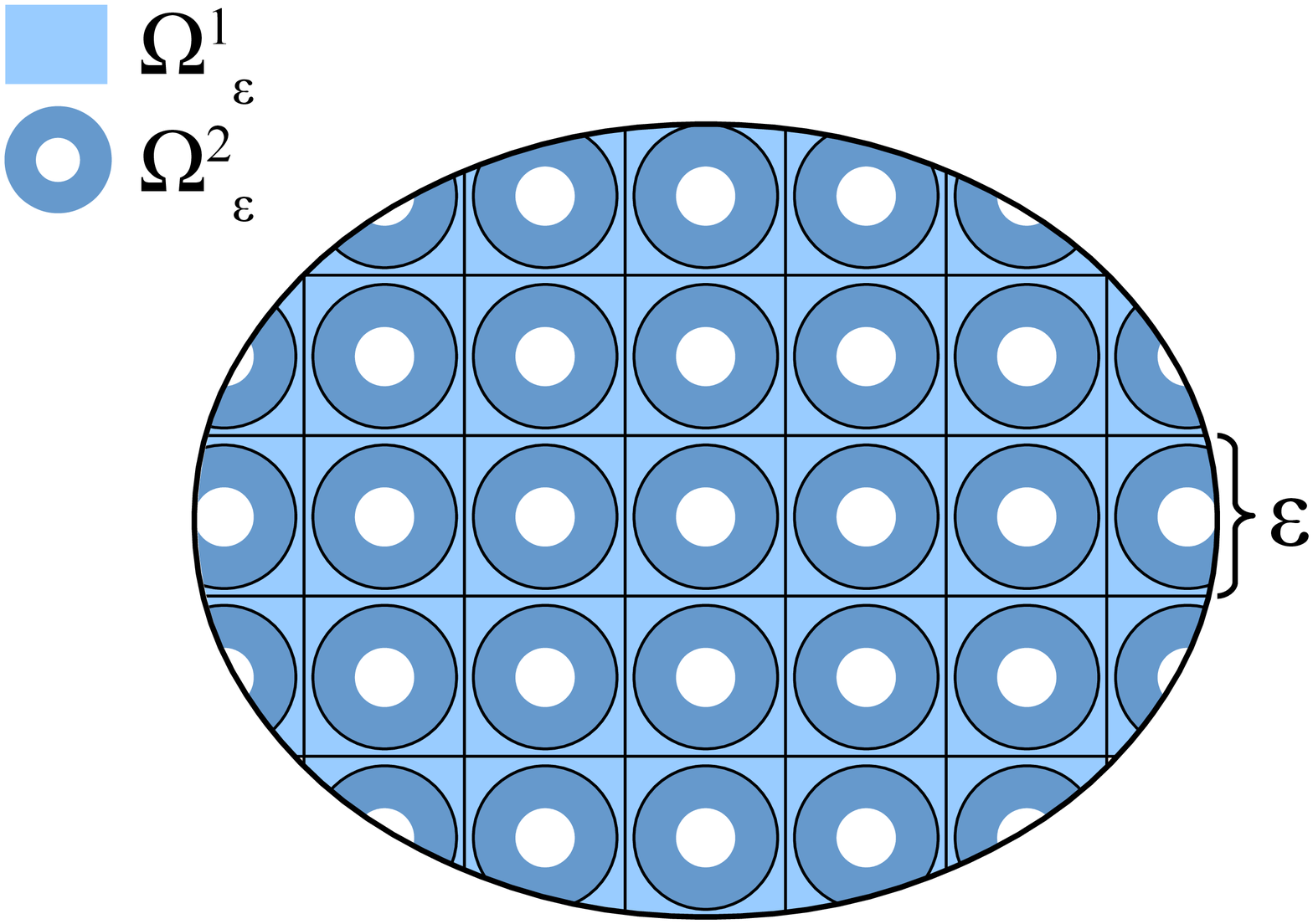}\label{fig:ref-cell-a}} 
  \subfigure[Reference cell
  $Y$.]{\raisebox{1cm}{\includegraphics[width=0.35\textwidth]{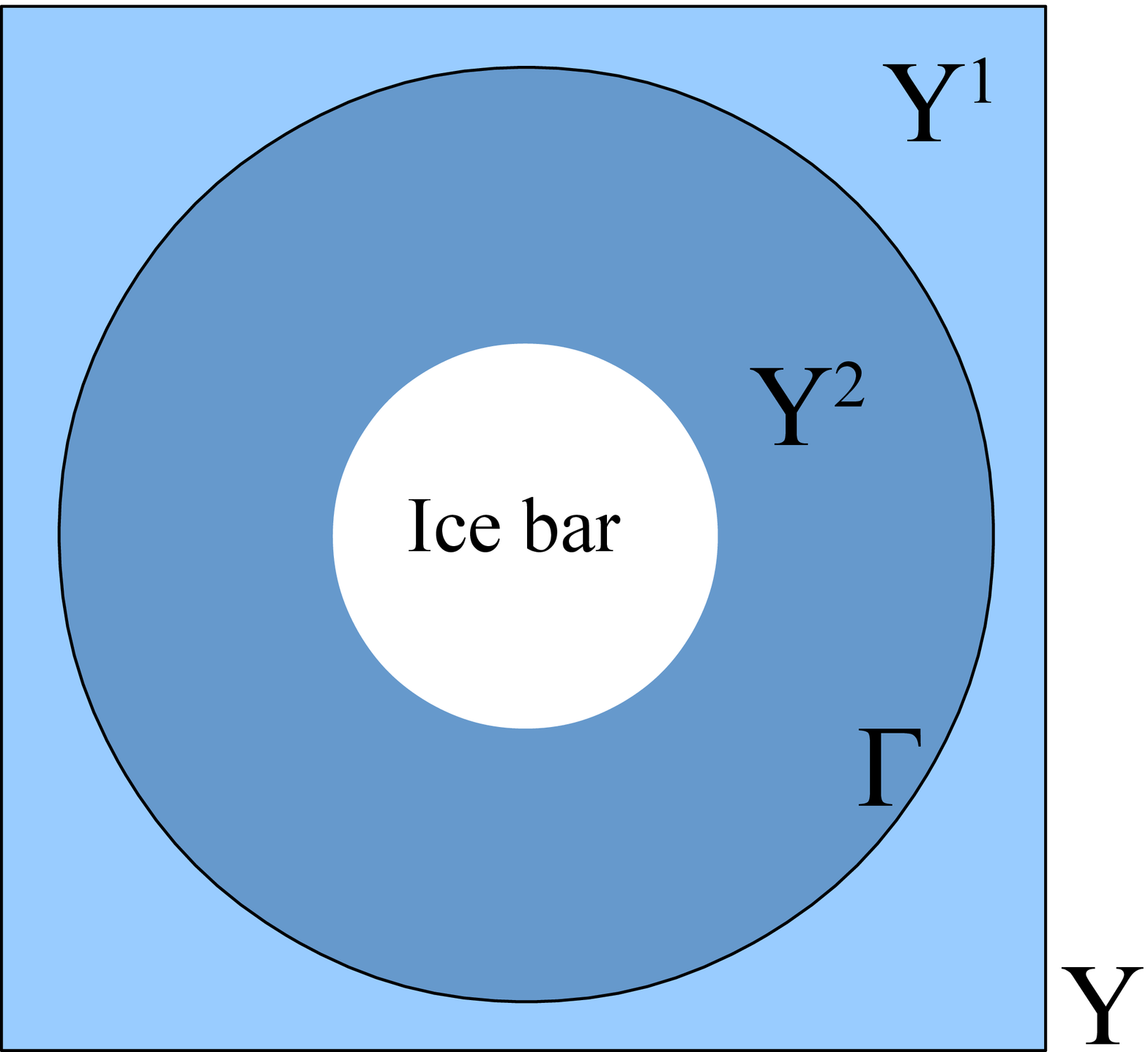}}\label{fig:ref-cell-b}} 
  \caption{Periodic microstructure of the reduced model with ice bars
    immersed in water, illustrating the decomposition into fast ($\Yone$
    and $\Omoneeps$) and slow ($\Ytwo$ and $\Omtwoeps$) diffusing
    regions.}
  \label{fig:ref-cell}
\end{figure}

Throughout the theoretical derivations in this paper, we employ what is
known as the two-phase formulation of the Stefan problem in which the
heat diffusion equation does not involve the time rate of change of
temperature but rather that of the specific enthalpy $\HE$, which
measures the heat energy stored in the water or ice per unit mass and
has units of $J/kg$.  Assuming that the material properties for water
and ice remain constant in their respective phase, the temperature $T$ (in
$\degK$) can be written as a piecewise linear function of
enthalpy~\cite{STEFAN5}
\begin{gather*}
  T = \widetilde{\omega}(\HE) = 
  \left\{\begin{array}{cl}
      \ds\frac{1}{c_i} \HE, & \qquad \text{if $\HE < \HE_i$}, \\[0.3cm]
      \Tcrit,               & \qquad \text{if $\HE_i \leq \HE < \HE_w$},\\[0.1cm]
      \ds\Tcrit + \frac{1}{c_w}(\HE-\HE_w),& \qquad \text{if $\HE_w \leq \HE$},
    \end{array}\right.
\end{gather*}
where $\Tcrit=273.15\degK$ is the melting/freezing point of water, $c$
denotes specific heat ($J/kg\, \degK$), and $\HE$ denotes the enthalpy
($J/kg$) at $T = \Tcrit$.  The subscripts $i$ and $w$ refer to the ice
and water phases, respectively.  A distinguishing feature of this
temperature-enthalpy relationship is that when the temperature is equal
to the melting point, the enthalpy varies while the temperature remains
constant -- this behavior derives from the fact that a certain amount of
energy (known as latent heat) is required to effect a change in
phase from solid to liquid at the interface between phases.

Because the function $\widetilde{\omega}(\HE)$ above is neither
differentiable nor invertible, we instead employ in our model a
regularized function $\omega(\HE)$ with ``rounded corners'' such as that
pictured in Figure~\ref{fig:omega-smooth}, which is incidentally a
better representation of what one would actually observe in a real
physical system:
\begin{gather}\label{def_omega}
  T = \omega(\HE) = 
  \left\{\begin{array}{cl}
      \ds\frac{1}{c_i} \HE,                      & \qquad \text{if $\HE < \HE_i^1$},\\[0.2cm]
      \text{\footnotesize[smooth connection]},& \qquad \text{if $\HE_i^1 \leq \HE < \HE_i^2$},\\[0.2cm]
      \ds\Tcrit-\frac{2\HE - (\HE_i^2 + \HE_w^1)}{2c_\infty},& \qquad \text{if $\HE_i^2 \leq \HE < \HE_w^1$},\\[0.2cm]
      \text{\footnotesize[smooth connection]},& \qquad \text{if $\HE_w^1 \leq \HE < \HE_w^2$},\\[0.2cm]
      \ds\Tcrit + \frac{1}{c_w}(\HE-\HE_w^2),    & \qquad \text{if $\HE_w^2 \leq \HE$}.
    \end{array}\right.
\end{gather}
The nearly flat portion of the temperature curve in the transition
region near $T\approx \Tcrit$ has a small positive slope
$\frac{1}{c_\infty}$, and the corners are smoothed out over the short
intervals $\HE_i^1 \lessapprox \HE_i \lessapprox \HE_i^2$ and $\HE_w^1
\lessapprox \HE_w \lessapprox \HE_w^2$. These smooth connections ensure
that $\omega$ is a continuously differentiable, invertible and monotone
increasing function of enthalpy.
\begin{figure}
  \centering
  \begin{tabular}{ccc}
    \includegraphics[height=0.2\textheight]{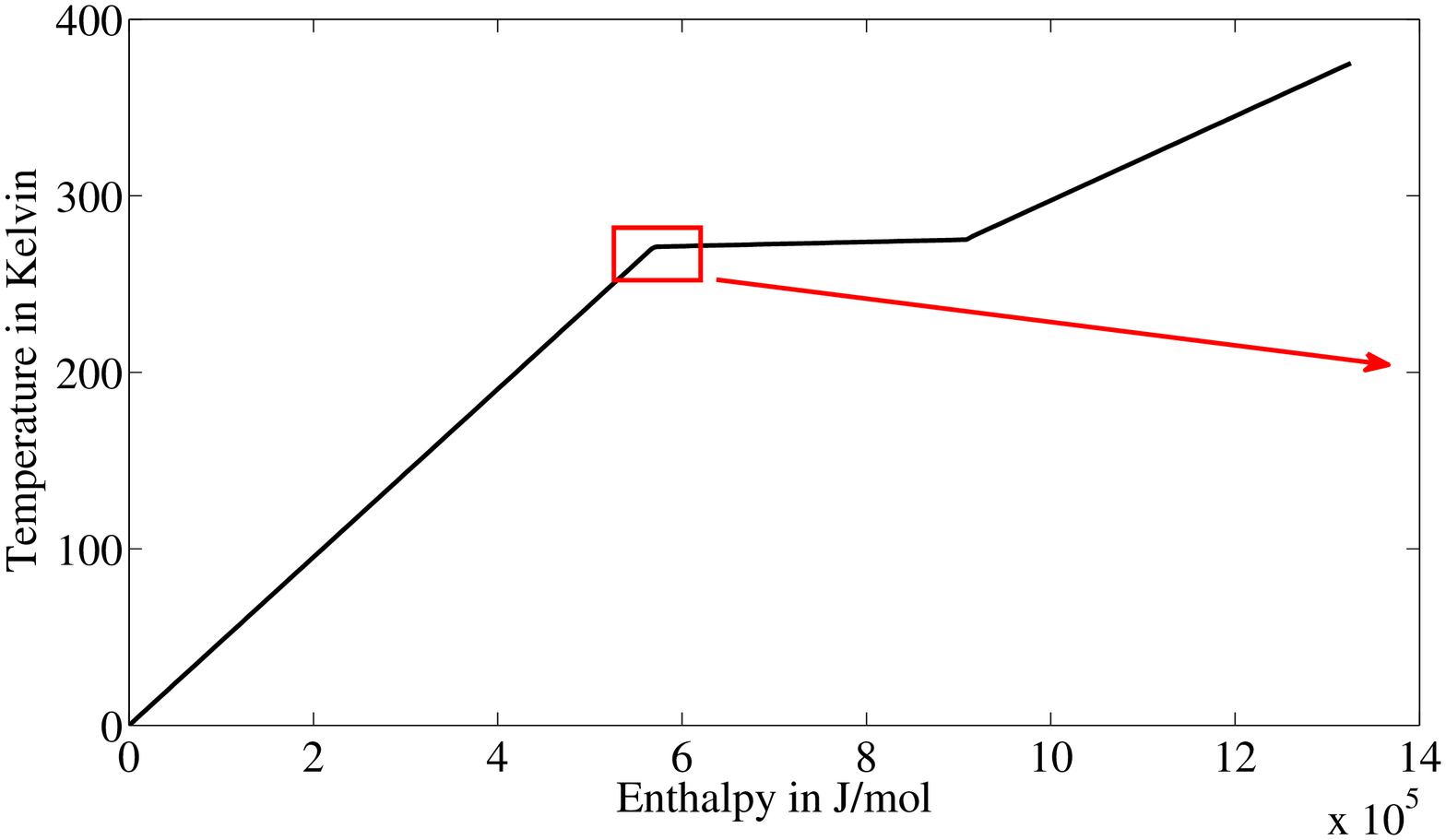}
    &\qquad&
    \includegraphics[height=0.2\textheight]{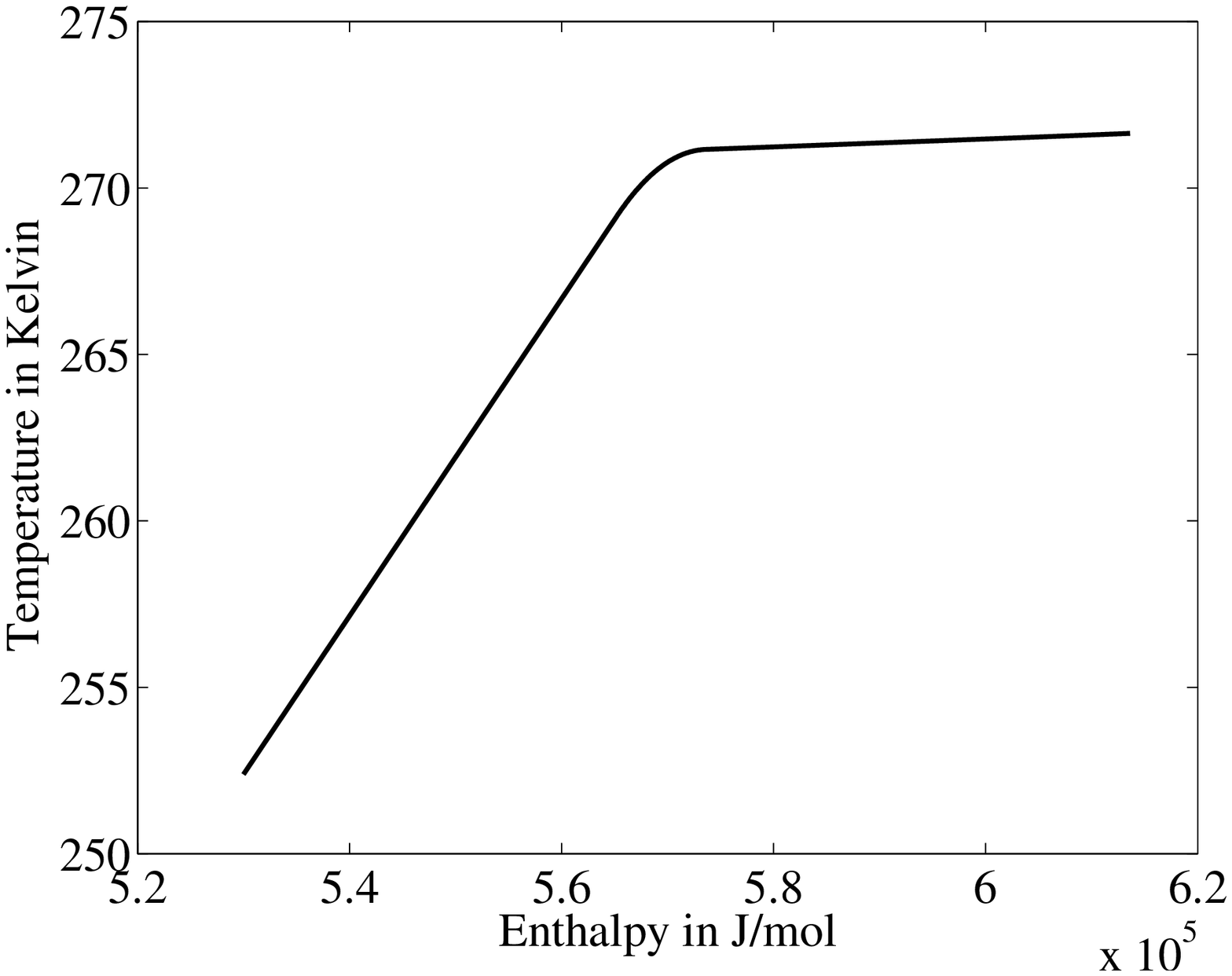}
  \end{tabular}
  \caption{Plot of the temperature-enthalpy function, $T=\omega(\HE)$,
    for a generic ice/water phase change problem in which the critical
    (melting) temperature is $\Tcrit=273.15\degK$.  The zoomed-in view
    on the right illustrates the smoothed corners in the
    regularization.}
  \label{fig:omega-smooth}
\end{figure}

We now describe the decomposition of the solution into slow and fast
diffusing variables.  The functions $\Toneeps$ and $\HEoneeps$ denote
the fast-diffusing temperature and enthalpy components (respectively)
and are defined on the sub-region $\Omega_\eps^1$. Similarly, $\Ttwoeps$
and $\HEtwoeps$ are the slow-diffusing temperature and enthalpy defined
on $\Omega_\eps^2$.  Next, we define the corresponding solution spaces
\begin{align*}
  \V^1_\eps &\defeq \{ u \in L^2([0,\tend], \Hil^1(\Omoneeps)) \cap
  \Hil^1([0,\tend], \Hil^1(\Omoneeps)')\; | \; u=0\text{ on
  }\partial\Omoneeps\cap\partial\Omega \},\\
  \V^2_\eps &\defeq \{ u \in L^2([0,\tend], \Hil^1(\Omtwoeps)) \cap
  \Hil^1([0,\tend], \Hil^1(\Omtwoeps)')\; | \; u=0\text{ on }\Gameps \},
\end{align*}
where the ``prime'' is used to denote a dual space and $[0,\tend]$
represents the time interval of interest for some fixed $\tend>0$.  The
corresponding test spaces are $V^1_\eps=\Hil^1(\Omoneeps)$ and
$V^2_\eps=\Hil^1_0(\Omtwoeps)$.  We also need to introduce notation for
inner products, with $(u,v)_{\Omega_\eps^\alpha} =
\int_{\Omega_\eps^\alpha} uv\,\text{d}x$ representing the inner product
of two functions in $\V^\alpha_\eps$ for $\alpha=1,2$, whereas
$(u,v)_{\Omega_\eps^\alpha,t} = \int_0^{t}\int_{\Omega_\eps^\alpha}
uv\,\text{d}x\, \text{d}\tau$ denotes that an additional time
integration is performed over the interval $[0,t]$ with $0\leq t\leq
\tend$.  Finally, we let $\langle u,v\rangle_{\Gamma_\eps} =
\int_{\Gamma_\eps}g_\eps uv\,\text{d}\sigma_x$ denote the inner product
on the interface $\Gamma_\eps$, where $g_\eps$ represents the Riemann
curvature tensor.

We are now prepared to state the weak form of the heat diffusion
problem.  Assuming that initial values $T_{1,\eps,init} =
\omega(\HE_{1,\eps,init})$ and $T_{2,\eps,init} =
\omega(\HE_{2,\eps,init})$ are smooth, non-negative and bounded, and
that a Dirichlet condition $\Toneeps = \Tout$ is imposed at the outer
boundary $\partial\Omega\cap\partial\Omega_\eps^1$, our goal is to find 
$(\Toneeps, \Ttwoeps) \in (\V^1+\Tout) \times (\V^2+\Toneeps)$ such that
\begin{subequations}\label{s_problem_weak}
  \begin{align}
    (\partial_t \HEoneeps,\; \phi)_{\Omoneeps} + (D(\HEoneeps)\nabla
    \Toneeps,\; \nabla\phi)_{\Omoneeps} 
    + \eps^2 \langle D(\HEtwoeps) \nabla \Ttwoeps\cdot \normvec,\;
    \phi\rangle_{\Gameps} &= 0, \label{s_problem_weak_1}\\  
    (\partial_t \HEtwoeps,\; \psi)_{\Omtwoeps} + (\eps^2
    D(\HEtwoeps)\nabla \Ttwoeps,\; \nabla\psi)_{\Omtwoeps} 
    &= 0,
    \label{s_problem_weak_2}
  \end{align}
\end{subequations}
for all $\phi, \psi \in V^1\times V^2$.  Note that $\normvec$ represents
the outward-pointing unit normal vector on $\Omtwoeps$, and that the
temperature and enthalpy are connected via $\Taeps = \omega(\HEaeps)$,
or equivalently $\HEaeps=\omega^{-1}(\Taeps)$.  Following
\cite{STEFAN5}, we take the diffusion coefficient to be a piecewise
affine linear function of enthalpy
\begin{align}\label{DH}
  D(\HE) = \left\{\begin{array}{cl}
      \frac{k_i}{\rho_i}, & \qquad \text{if } \HE<\HE_i,\\
      \frac{k_i}{\rho_i} + \frac{\HE-\HE_i}{\HE_w-\HE_i}
      \left(\frac{k_w}{\rho_w} - \frac{k_i}{\rho_i}\right), 
      & \qquad \text{if } \HE_i\leq \HE < \HE_w, \\ 
      \frac{k_w}{\rho_w}, & \qquad \text{if } \HE_w < \HE,\\ 
    \end{array}\right.
\end{align}
where $\rho_w$ and $\rho_i$ are water and ice densities ($kg/m^3$),
$k_w$ and $k_i$ are thermal conductivities ($W/m \, \degK$), and $\HE_w$
and $\HE_i$ are the corresponding enthalpies at $T=\Tcrit$.  We note in
closing that slow diffusion is induced in our problem by the factor
$\eps^2$ multiplying terms in Eqs.~\eqref{s_problem_weak} that involve
the diffusion coefficient $D(\HEtwoeps)$.

The following four sections contain the primary theoretical results
relating to solutions of the reduced problem in the weak form.

\subsection{Existence}
\label{sec:existence}

To prove the existence of a solution to \eqref{s_problem_weak} for every
$\eps>0$, we employ a result of Showalter \cite[Prop.~4.1
($p=2$)]{PDE5}.  Let $V$ represent a reflexive Banach space with dual
$V'$ and let $\V = L^2([0,\tend],V)$ with dual $\V' =
L^2([0,\tend],V')$.  We also have the Hilbert space $H$ with $V\subset
H\subset V'$, where $V$ is dense and continuously embedded in $H$.


\subsubsection{General function space setting}

Suppose that $\A:V\rightarrow V'$ is a given function (not necessarily
linear) where $u_0\in H$ and $f\in \V'$. We are interested in finding 
$u\in\V$ such that
\begin{align}\label{problem_existence}
  \int_0^\tend \partial_tu(t)v(t)\,\text{d}t 
  + \int_0^\tend \A(u(t))v(t)\,\text{d}t 
  = \int_0^\tend f(t)v(t)\,\text{d}t, 
\end{align}
with $u(0) = u_0\in H$ for every $v\in \V$. 
That is, $\V$ is the solution and the test space.

\begin{definition}
  An operator $\A$ is hemicontinuous if for each $u,\,v\in V$, the
  real-valued function $\beta\mapsto\A(u+\beta v)v$ is continuous.
\end{definition}

\begin{propShowalter}
  Let the spaces $V,H,\V$ be defined as above and assume that $V$ is
  separable.  Consider the family of operators $\A(t,\cdot):V\rightarrow
  V'$ for $0\leq t\leq \tend$ that satisfy the following conditions:
  \begin{enumerate}
  \item[(i)] for each $v\in V$ the function
    $\A(\cdot,v):[0,\tend]\rightarrow V'$ is measurable;
  \item[(ii)] for almost every $t\in[0,\tend]$ the operator
    $\A(t,\cdot):V\rightarrow V'$ is hemicontinuous and bounded
    according to
    \begin{gather*}
      \norm{\A(t,v)}\leq C(\norm{v} + g(t)),
    \end{gather*}
    for all $v\in V$ and $g\in L^2[0,\tend]$;
  \item[(iii)] there are real constants $\alpha,\lambda>0$ such that 
    \begin{gather*}
      \A(t,v)v + \lambda \norm{v}_H^2\geq \alpha\norm{v}^2_V,
    \end{gather*}
    for almost every $t\in[0,\tend]$ and $v\in V$.
  \end{enumerate}
  Then, for every $f\in \V'$ and $u_0\in H$, there exists $u\in \V$ such that
  \begin{gather*}
    \partial_t u(t) + \A(t,u(t)) = f(t) \; \in\; \V' 
    \quad \text{and} \quad u(0)=u_0\ \in\; H.
  \end{gather*}
\end{propShowalter}
We remark that in~\cite{PDE5}, Showalter imposed the additional
assumption that operator $\A$ is monotone; however, this condition
was only required for uniqueness of the solution, and we do not claim
uniqueness here.

\subsubsection{Application to reduced model, with existence result}
\label{sec:setting-id}

For our simple reduced model that captures heat diffusion in a periodic
array of ice bars immersed in water, the space $V$ is equal to
$\Hil^1(\Omoneeps)\times \Hil^1(\Omtwoeps)$ and hence 
the solution and test space $\V = L^2([0,\tend],\Hil^1(\Omoneeps)) \times
L^2([0,\tend],\Hil^1(\Omtwoeps))$.  We also have that $f\equiv 0\in \V'$
and $u_0 = (\HEoneeps(0),\HEtwoeps(0))$.  The operator $\A$
corresponding to the weak form of the problem \eqref{s_problem_weak} is
\begin{multline}
  \A u(v)  = \left(D(u_1)\omega'(u_1)\nabla u_1,\;\nabla
    v_1\right)_{L^2(\Omega_\eps^1)} + \eps^2\left\langle
    D(u_2)\omega'(u_2)\nabla u_2 \cdot \normvec,\;
    v_1\right\rangle_{\Hil^{-\frac12}(\Gamma_\eps)\times
    \Hil^{\frac12}(\Gamma_\eps)} \\ 
  + \eps^2\left(D(u_2)\omega'(u_2)\nabla u_2,\; \nabla
    v_2\right)_{L^2(\Omega_\eps^2)} 
\end{multline}
where $u=(u_1,u_2)$ and $v=(v_1,v_2)$ are both in $V$.  Notice that $\A$
doesn't depend explicitly on time $t$ and consequently hypothesis~(i) in
Showalter's Proposition is trivially satisfied.  We may then restate the
problem \eqref{problem_existence} as one of finding
$u=(\HEoneeps,\HEtwoeps)\in(\V^1+\omega^{-1}(\Tout)) \times
(\V^2+\HEoneeps)$ such that
\begin{multline}\label{problem_existence_1}
  \int_0^{\tend} (\partial_t\HEoneeps,\; \phi)_{\Omega_\eps^1} +
  (\partial_t\HEtwoeps,\; \psi)_{\Omega_\eps^2}\, \text{d}t  
  + \int_0^{\tend}\left(D(\HEoneeps)\omega'(\HEoneeps)\nabla \HEoneeps,\;
    \nabla \phi\right)_{L^2(\Omega_\eps^1)} \\ 
  + \eps^2\left\langle D(\HEtwoeps)\omega'(\HEtwoeps)\nabla \HEtwoeps
    \cdot \normvec,\; \phi\right\rangle_{\Hil^{-\frac12}(\Gamma_\eps)\times
    \Hil^{\frac12}(\Gamma_\eps)} \\ 
  + \eps^2\left(D(\HEtwoeps)\omega'(\HEtwoeps)\nabla \HEtwoeps,\; \nabla
    \psi\right)_{L^2(\Omega_\eps^2)}\, \text{d}t = 0 
\end{multline}
for all $\phi, \psi \in V^1\times V^2$.  To apply Showalter's
Proposition, we take $\HEoneeps$ and $\HEtwoeps$ equal to solution $u$
and therefore adjust the function spaces by adding the Dirichlet
conditions.  We thereby obtain the differential operator $\A$ used in our
model, which we can show satisfies the remaining two conditions in
Showalter's Proposition.  We state this result as a Lemma whose proof
is provided in Appendix~\ref{app:existence}.


\begin{lemma}\label{lemma_existence}
 The operator $\A:V\rightarrow V'$ is
 \begin{enumerate}[label=\alph*)]
  \item continuous,
  \item bounded, and
  \item coercive, in the sense that there are real constants
    $\lambda,\alpha>0$ such that 
    \begin{gather*}
      \A(t,v)(v) + \lambda \norm{v}_H^2\geq \alpha\norm{v}^2_V,
    \end{gather*}
    for almost every $v\in V$ and $t\in [0,\tend]$.
  \end{enumerate}
\end{lemma}

From continuity it follows immediately that $\A$ is hemicontinuous.  
Therefore, we can apply Showalter's Proposition to obtain existence of
a solution to the system \eqref{problem_existence_1}.

\subsection{A~priori estimates}
\label{sec:a-priori}

Our main results on a~priori estimates are stated in
Lemma~\ref{lemma_estimates} below, in which the functions $\HEoneeps$,
$\Toneeps$, and $\HEtwoeps$ and $\Ttwoeps$ are bounded in
$L^2([0,\tend],\Hil^1(\Omoneeps))$ and
$L^2([0,\tend],\Hil^1(\Omtwoeps))$ independent of $\eps$ (part a).
Furthermore, $D(\HEtwoeps)\nabla\Ttwoeps$ must be bounded independent of
$\eps$ in $L^2([0,\tend],L^2(\Gameps))$ (parts b,c) and there is also a
strong convergence result for $\HEoneeps$ and $\HEtwoeps$ in
$L^2([0,\tend],L^2(\Omtwoeps))$ and $L^2([0,\tend],L^2(\Omoneeps))$
respectively (parts d,e,f).  The proof of this Lemma is given
Appendix~\ref{app:estimates}. 

\begin{lemma}\label{lemma_estimates}
Here, the notation $C_i$ represents a positive constant that is independent of
$\eps$. 
  \begin{enumerate}[label=\alph*)]
  \item  There exists a constant $\bigC_1$ such that 
    \begin{gather*}
      \Norm{\Toneeps}^2_{\Omoneeps} + \Norm{\nabla
        \Toneeps}^2_{\Omoneeps} + \Norm{\Ttwoeps}^2_{\Omtwoeps}
      + \eps^2\Norm{\nabla \Ttwoeps}^2_{\Omtwoeps} \leq \bigC_1  .
    \end{gather*}
    
  \item  There exists another constant $\bigC_2$ such that
    \begin{multline*}
      \Norm{\nabla \HEoneeps}^2_{\Omoneeps} +
      \Norm{\nabla\cdot[D(\HEoneeps)\nabla \Toneeps]}^2_{\Omoneeps,t} + 
      \eps^2\Norm{\nabla\HEtwoeps}^2_{\Omtwoeps} + \eps^4\Norm{\nabla
        \cdot [D(\HEtwoeps)\nabla \Ttwoeps}^2_{\Omtwoeps,t} \leq \bigC_2.
    \end{multline*}
    
  \item There exists a constant $\bigC_3$
    such that
    \begin{gather*}
      \eps^3\Norm{D(\HEtwoeps)\nabla \Ttwoeps}^2_{\Gameps,t} \leq
      \bigC_3. 
    \end{gather*}
    
    
  \item The functions $\HEoneeps$, $\HEtwoeps$ are nonnegative for almost
    every $x\in\Omoneeps$, $\Omtwoeps$ as long as $\HEoneeps(0)$,
    $\HEtwoeps(0)$ are nonnegative (respectively).
    
  \item There exists a constant $\bigC_4$
    such that
    \begin{gather*}
      \Norm{\HEoneeps}_{L^\infty(\Omoneeps)} +
      \Norm{\HEtwoeps}_{L^\infty(\Omtwoeps)} \leq \bigC_4.
    \end{gather*}
    
  \item There exists a constant $\bigC_5$
    such that 
    \begin{gather*}
      \Norm{\partial_t\HEoneeps}_{L^2([0,\tend],L^2(\Omoneeps))} +
      \Norm{\partial_t \HEtwoeps}_{L^2([0,\tend],L^2(\Omtwoeps))}
      \leq \bigC_5.
    \end{gather*}
  \end{enumerate}
\end{lemma}

Taking Lemma~\ref{lemma_estimates} along with the extension theorem
in~\cite{TWOSCALE22}, we deduce that $\HEoneeps$ and $\HEtwoeps$ are
elements of $L^2([0,\tend],\Hil^1(\Omega))\cap
\Hil^1([0,\tend],\Hil^{-1}(\Omega))\cap
L^\infty([0,\tend]\times\Omega)$.  An application of Lemma~5.6
of~\cite{TWOSCALE21} then yields that $\HEoneeps$ and $\HEtwoeps$
converge strongly in $L^2([0,\tend],L^2(\Omega))$ to some functions
$\HE_{1,0}$ and $\HE_{2,0}$ respectively.  Because $D$ is a continuous
and bounded function, it then follows from the Theorem of
Nemytskii~\cite[p.~48]{PDE5} that the function $D : \HE_{i,\eps} \mapsto
D(\HE_{i,\eps})$ is also continuous and bounded for $i=1,2$. Therefore,
we conclude that $D(\HEoneeps)$ and $D(\HEtwoeps)$ converge strongly in
$L^2([0,\tend],L^2(\Omega))$ to $D(\HE_{1,0})$ and $D(\HE_{2,0})$,
respectively.

Finally, because $\Toneeps = \omega(\HEoneeps)$ where $\omega$ is
smooth, continuous and bounded (and similarly for $\omega^{-1}$), we
obtain that $\HEoneeps$ and $\Toneeps$ both converge strongly in
$L^2([0,\tend],L^2(\Omega))$ and that $T_{1,0} = \omega(\HE_{1,0})$.  An
analogous result holds for $\Ttwoeps = \omega(\HEtwoeps)$.

\subsection{Identification of the two-scale limit}
\label{sec:two-scale}

Armed with the fact that a solution of \eqref{s_problem_weak} exists for
every $\eps>0$, and that the {a~priori} estimates from the previous
section hold, we know that the terms in Eqns.~\eqref{s_problem_weak}
two-scale converge to some limit (see \cite{TWOSCALE1,
  TWOSCALE6}).  In order to investigate this limit, we define 
test functions that vary on length scales of size $O(1)$ and
$O(\eps)$ according to
\begin{gather*}
  \phi_\eps\left(x,\textstyle\frac{x}{\eps}\right) = \phi_0(x) +
  \eps\phi_1\left(x,\textstyle\frac{x}{\eps}\right)
  \qquad \text{and} \qquad 
  \psi_\eps\left(x,\textstyle\frac{x}{\eps}\right), 
\end{gather*}
where $(\phi_0,\phi_1) \in C^\infty(\Omega)\times
C^\infty(\Omega,C^\infty_\#(Y))$ and $\psi_\eps\in
C^\infty(\Omega,C^\infty_\#(Y))$, and the subscript $\#$ denotes
periodicity in space.
We also need to introduce the characteristic functions
\begin{gather*}
  \chi^i(\textstyle\frac{x}{\eps}) = 
  \begin{cases}
    1, & \text{if $\frac{x}{\eps}\in \Omeps^i$,}\\
    0, & \text{otherwise},    
  \end{cases}
\end{gather*}
for $i=1,2$.  By substituting $\phi_\eps$ into \eqref{s_problem_weak_1}
and using $\chi^i$ to write the resulting integrals over the entire
domain $\Omega$, we obtain
\begin{multline*}
  \int_{\Omega}\chi^1
  \left(\textstyle\frac{x}{\eps}\right) \partial_t\HEoneeps(x,t)
  \phi_\eps \left(x,\textstyle\frac{x}{\eps}\right) \, \text{d}x +
  \int_{\Omega}\chi^1 \left(\textstyle\frac{x}{\eps}\right)D(\HEoneeps)
  \nabla \Toneeps(x,t) \nabla\phi_\eps
  \left(x,\textstyle\frac{x}{\eps}\right) \,\text{d}x\\
  + \eps^2\int_{\Gamma_\eps}D(\HEtwoeps)\nabla \Ttwoeps(x,t)\cdot
  \normvec\phi_\eps \left(x,\textstyle\frac{x}{\eps}\right)\,
  \text{d}\sigma_x = 0.
\end{multline*}
Then, taking the limit as $\eps\to 0$ yields
\begin{multline}
  \label{eq:limit1}
  \int_{\Omega\times Y^1}\partial_t \HE_{1,0}(x,t)\phi_0(x)\,
  \text{d}y\, \text{d}x \\
  + \int_{\Omega\times Y^1}D(\HE_{1,0})[\nabla_xT_{1,0}(x,t) + \nabla_y
  \widehat{T}_{1,0}(x,y,t)]
  [\nabla_x\phi_0(x) + \nabla_y\phi_1(x,y)]\, \text{d}y\, \text{d}x\\
  + \int_{\Omega\times\Gamma}D(\HE_{2,0})\nabla_yT_{2,0}(x,y,t)\cdot
  \normvec\phi_0(x)\, \text{d}\sigma_y\, \text{d}x = 0,
\end{multline}
where $y$ denotes the variable that is defined on the reference cell $Y$.
Note that
$T_{1,0} = \omega(\HE_{1,0})\in L^2([0,\tend],\Hil^1(\Omega))$ is
independent of $y$, and we have also introduced $\widehat{T}_{1,0}\in
L^2([0,\tend],L^2(\Omega,\Hil^1_\#(Y^1)))$ and
$T_{2,0}=\omega(\HE_{2,0})\in
L^2([0,\tend],L^2(\Omega,\Hil^1_\#(Y^2)))$. The limit of $\nabla
T_{1,\eps}$ has a special form proven in \cite{TWOSCALE1} that consists
of two terms: one involving a gradient with respect to the slow
variable, and a second term with respect to the fast variable.  In a
similar manner, we substitute the second test function $\psi_\eps$ into
\eqref{s_problem_weak_2} to get
\begin{align*}
  \int_{\Omega}\chi^2
  \left(\textstyle\frac{x}{\eps}\right) \partial_t\HEtwoeps(x,t)
  \psi_\eps \left(x,\textstyle\frac{x}{\eps}\right)\, \text{d}x +
  \eps^2\int_{\Omega}\chi^2 \left(\textstyle\frac{x}{\eps}\right)
  D(\HEtwoeps)\nabla \Ttwoeps(x,t) \nabla\psi_\eps
  \left(x,\textstyle\frac{x}{\eps}\right)\, \text{d}x=0,
\end{align*}
and again take the limit as $\eps\to 0$ to obtain
\begin{align}
  \label{eq:limit2}
  \int_{\Omega\times Y^2}\partial_t\HE_{2,0}(x,y,t)\psi_0(x,y)\,
  \text{d}y\, \text{d}x +\int_{\Omega\times Y^2}
  D(\HE_{2,0})\nabla_yT_{2,0}(x,y,t) \nabla_y\psi_0(x,y)\,
  \text{d}y\, \text{d}x = 0. 
\end{align}

We are free at this point to choose any test function, and so now we set
$\phi_0 = 0$ in Eq.~\eqref{eq:limit1} and then introduce functions
$\mu_i\in \Hil^1_\#(Y^1)$ in order to express $\widehat{T}_{1,0}(x,y,t)
= \sum_{i=1}^n\partial_{x_i}T_{1,0}(x,t) \mu_i(y)$ in separable form.
The weak formulation of the cell problem for $i=1,\dots,n$ may then be
expressed in the much simpler form
\begin{align}
  \label{cell_problem1}
  (e_i + \nabla_y\mu_i, \; \nabla_y\phi_1)_{\Yone} = 0,
\end{align}
where $\mu_i$ is $Y$-periodic.  Alternately, we may take $\phi_1 = 0$ in
Eq.~\eqref{eq:limit1} to obtain 
\begin{multline*}
  \int_{\Omega\times Y^1}\partial_t \HE_{1,0}(x,t)\phi_0(x)\,
  \text{d}y\, \text{d}x 
  + \int_{\Omega\times Y^1}D(\HE_{1,0})
  \sum_{i=1}^n\partial_{x_i}T_{1,0}(x,t)[e_i + \nabla_y 
  \mu_i(y)]\nabla_x\phi_0(x)\, \text{d}y\, \text{d}x\\
  + \int_{\Omega\times\Gamma}D(\HE_{2,0})\nabla_yT_{2,0}(x,y,t)\cdot
  \normvec\phi_0(x)\, \text{d}\sigma_y\, \text{d}x = 0,
\end{multline*}
which can be rewritten in the more suggestive form
\begin{multline}\label{eq:limit3}
  \int_{\Omega\times Y^1}\partial_t \HE_{1,0}(x,t)\phi_0(x)\,
  \text{d}y\, \text{d}x 
  + \int_{\Omega}D(\HE_{1,0}) \sum_{i,j=1}^n\partial_{x_i}T_{1,0}(x,t)
  \int_{Y^1}[\delta_{ij} + \partial_{y_j} 
  \mu_i(y)]\, \text{d}y\, \partial_{x_j}\phi_0(x)\, \text{d}x\\
  + \int_{\Omega\times\Gamma}D(\HE_{2,0})\nabla_yT_{2,0}(x,y,t)\cdot
  \normvec\phi_0(x)\, \text{d}\sigma_y\, \text{d}x = 0. 
\end{multline}
The diffusion term involves the extra factors
\begin{align}
  \label{cell_problem2}
  \Pi_{ij} = \int_{\Yone}(\delta_{ij} + \partial_{y_j}\mu_i)
  \,\text{d}y 
\end{align}
for $i,j=1,\dots,n$ which can be represented as a matrix $\Pi$ that
multiplies the diffusion coefficient $D(\HE_{1,0})$.  Hence, we obtain
the limit equations from \eqref{eq:limit2} and \eqref{eq:limit3}
\begin{subequations}\label{s_problem_limit_1}
  \begin{align}
    \abs{Y^1}(\partial_t\HE_{1,0},\; \phi)_{\Omega} + (\Pi
    D(\HE_{1,0})\nabla_x T_{1,0},\; \nabla_x\phi)_{\Omega} +
    (D(\HE_{2,0})\nabla_y T_{2,0}\cdot \normvec,\; \phi)_{\Omega\times\Gamma}
    &=\ 0, 
    \label{s_problem_limit_1a}\\ 
    (\partial_t\HE_{2,0},\; \psi)_{\Omega\times Y^2} +
    (D(\HE_{2,0})\nabla_yT_{2,0},\; \nabla_y\psi)_{\Omega\times Y^2}
    &=\ 0, \label{s_problem_limit_1b} 
  \end{align}
\end{subequations}
for all $\phi\in \Hil^1(\Omega)$ and $\psi\in
L^2(\Omega,\Hil^1_\#(\Ytwo))$, where $\Pi$ is the $n\times n$ matrix of
scaling factors defined in \eqref{cell_problem2}, $T_{1,0}\in
L^2([0,\tend],\Hil^1(\Omega))$ with $T_{1,0}=\Tout$ on $\partial\Omega$,
and $T_{2,0}\in L^2([0,\tend],L^2(\Omega,\Hil^1_\#(\Ytwo)))$ with
$T_{2,0}=T_{1,0}$ on $\Omega\times\Gamma$.

To simplify notation in the remainder of the paper, we drop the zero
subscripts in $(T_{1,0},\HE_{1,0}, T_{2,0},\HE_{2,0})$ and denote them
instead by $(T_1,\HE_1,T_2,\HE_2)$.

\subsection{Uniqueness}
\label{sec:unique}

The uniqueness of the solution to equations \eqref{s_problem_limit_1}
may be formulated compactly in terms of the following lemma, which is
proven in Appendix~\ref{app:uniqueness}.

\begin{lemma}\label{uniqueness}
  Equations~\eqref{s_problem_limit_1} have at most one solution given by
  \begin{alignat*}{3}
    T_1 &\in \V^1(\Omega) + \Tout &&= L^2([0,\tend], \Hil^1_0(\Omega))
    \cap \Hil^1([0,\tend], L^2(\Omega)) + \Tout, \\  
    T_2 &\in \V^2(\Omega\times \Ytwo) + T_1 && = 
		L^2([0,\tend], L^2(\Omega,\Hil^1_\#(\Ytwo))) \cap 
		\Hil^1([0,\tend], L^2(\Omega\times\Ytwo)) + T_1, 
  \end{alignat*}
  where $T_1 = \omega(\HE_1)$ and $T_2 = \omega(\HE_2)$.
\end{lemma}

\subsection{Strong formulation of the limit problem}
\label{sec:strong_form}

We now state an equivalent strong formulation of the limit problem
corresponding to the weak form in Eqs.~\eqref{s_problem_limit_1}, which
consists of a PDE for $\HE_1$ on the entire domain $\Omega$ 
\begin{alignat*}{3}
  \abs{Y^1}\partial_t\HE_1 - \nabla_x\cdot(\Pi D(\HE_1)\nabla_xT_1) =
  \int_\Gamma D(\HE_2)\nabla_yT_2\cdot \normvec \, \text{d}\sigma_y
  \qquad&&\text{in}\ \Omega, \\
  T_1 = \Tout \qquad &&\text{on}\ \partial\Omega, \\
  \intertext{along with a second PDE for $\HE_2$ on 
    $\Omega \times Y^2$}
  \partial_t\HE_2 - \nabla_y\cdot (D(\HE_2)\nabla_yT_2) = 0 \qquad
  &&\text{on}\ \Omega\times Y^2, \\  
  T_2 = T_1\qquad &&\text{on}\ \Omega\times \Gamma, 
\end{alignat*}
along with suitable initial values for enthalpy that we denote
$\HE_{1,init}$ and $\HE_{2,init}$.  These two problems are coupled
through the temperature function $T_2$ via an integral term and the
final boundary condition, both of which are enforced on $\Gamma$.  This
is known as the \emph{two-phase formulation of the Stefan problem}, for
which there is no explicit equation for the motion of the phase
interface; instead, the interface location is captured implicitly
through the function $T=\omega(\HE)$.

Under suitable assumptions (see~\cite{STEFAN5}) the above problem may be
rewritten in an equivalent \emph{one-phase formulation} consisting of
the same equation for the enthalpy $\HE_1$ on the macroscale
\begin{subequations}
  \label{limit_problem_strong1}
  \begin{alignat}{3}
    \label{limit_problem_strong1a}
    \abs{Y^1}\partial_t\HE_1 - \nabla_x\cdot(\Pi D(\HE_1)\nabla_xT_1) =
    \int_\Gamma D(\HE_2)\nabla_yT_2\cdot \normvec\, \text{d}\sigma_y
    \qquad &&\text{in}\ \Omega,\\ 
    \label{limit_problem_strong1b}
    T_1 = \Tout \qquad &\qquad\qquad&\text{on}\ \partial\Omega,
  \end{alignat}
\end{subequations}
along with the following equations for temperature on the reference cell
\begin{subequations}
  \label{limit_problem_strong2}
  \begin{alignat}{3}
    \label{limit_problem_strong2a}
    \qquad\qquad\qquad\qquad\qquad c_w\partial_tT_2 - \nabla_y\cdot
    (D(\HE_2)\nabla_yT_2) = 
    0\qquad &&\text{on}\ \Omega\times \tilde{Y}^2(x,t),\\ 
    \label{limit_problem_strong2b}
    T_2 = T_1\qquad &&\text{on}\ \Omega\times \Gamma,\\ 
    \label{limit_problem_strong2c}
    T_2 = \Tcrit\qquad &&\text{on}\ \Omega\times\partial\tilde{Y}^2(x,t).
  \end{alignat}
\end{subequations}
These equations are supplemented by an evolution equation for the phase
boundary
\begin{alignat}{3}
  \label{limit_problem_strong3}
  \qquad\qquad\qquad\qquad\qquad\qquad\qquad \partial_t s =
  -\frac{D(\HE_2)}{(\HE_w-\HE_i)}\nabla_yT_2\cdot \normvec\qquad
  &&\text{on}\ \Omega\times\partial\tilde{Y}^2(x,t),
\end{alignat}
where the ice/water interface is denoted $\partial\tilde{Y}^2(x,t)$ and
depends on both $x$ and $t$ because the domains may differ in
$\Omega\times[0,\tend]$.  Note that in this one-phase formulation the
temperature is only defined within the portion of the domain that
contains water (for freezing) or ice (for thawing).  Furthermore, we
obtain an explicit equation \eqref{limit_problem_strong3} governing the
dynamics of the location $s(x,y,t)$ of the phase interface, which is
known as the \emph{Stefan condition}.  When performing numerical
simulations of the limit problem, it is most convenient to employ the
one-phase formulation of the Stefan problem.

\section{Numerical simulations of the homogenized problem}
\label{sec:sims-reduced}

\subsection{Description of the numerical algorithm}
\label{sec:sims-algo}

We now develop an algorithm for computing the solution $\HE_1(x,t)$,
$T_2(x,y,t)$ and $s(x,y,t)$ to the one-phase Stefan problem
\eqref{limit_problem_strong1}--\eqref{limit_problem_strong3}, which also
requires values of the coefficients $\Pi_{ij}$ from the reference cell
problem \eqref{cell_problem1}.  We consider only 2D simulations in this
paper, although the algorithm extends naturally to three dimensions.
Our approach is a simple one that is based on three key approximations:
\begin{itemize}
\item The coefficients $\Pi_{ij}$ are defined on the reference cell and
  are in fact constants that are independent of temperature and
  interface location.  Consequently, these coefficients only need to be
  computed once at the beginning of a simulation.
\item The reduced model has an inherent radial symmetry on the
  microscopic scale, and so we restrict ourselves in this paper to
  problems that have an analogous radial symmetry on the macroscale; 
  this is natural for the application to a tree cross-section
  considered in Section~\ref{sec:maple_sap_exudation}.  Consequently,
  all variables and governing equations are recast in terms of a single
  radial coordinate (which we continue to label as $x$ or $y$) so that
  only one-dimensional problems are solved at both the micro- and
  macroscale.
\item Because of the simple form of the coupling between microscale and
  macroscale problems that involves only interfacial solution values, we
  propose a simple ``frozen coefficient'' splitting approach for
  time-stepping in which each variable is advanced to the next time step
  individually, holding all other variables constant at their previous
  values.
\end{itemize}

With this in mind, we can now summarize our multiscale numerical
algorithm as follows.  We use a method-of-lines approach in which the
unknowns are discretized in space (radius) and then the resulting
coupled system of ordinary differential equations is integrated in time
using a stiff ODE solver.  In particular, we have implemented our
algorithm in {\sf Matlab} and employed the solver {\tt ode15s}.  Because
our ultimate goal is to simulate sap flow in a tree stem having a
circular cross-section, we consider a problem domain $\Omega$ that is a
circle with some given radius $\Rtree$.  The spatial discretization is
then performed on two distinct scales:
\begin{itemize}
\item the \emph{macroscale} corresponding to the circular domain
  $\Omega$, that is divided into an equally-spaced grid of $\Mmacro$
  radial points at each of which we have unknown values of $\HE_1$,
  $T_1$ and $s$; and
\item the \emph{microscale} corresponding to the reference cell $Y$,
  which is solved for $T_2$ and $E_2$ at each macroscale point.  The
  local ice region $\tilde{Y}^2$ grows or shrinks in time according to
  the local phase boundary location $s$.  Therefore, we use a
  \emph{moving mesh} approach wherein the reference cell is discretized
  at a set of $\Mmicro$ equally-spaced radial points that move in time
  as the ice/water phase interface moves.  In practice, it is sufficient
  to use a coarse grid here with $\Mmicro \ll \Mmacro$.
\end{itemize}
The algorithm now proceeds according to the following steps.
\begin{description}
\item[\sf Step 1:] Set initial values of $T_2=T_{2,init}$ at each
  macroscale point, while $\HE_1=\HE_{1,init}$ and $s=s_{init}$ within
  each local reference cell.
\item[\sf Step 2:] For a single canonical reference cell that takes the
  shape of a square with a circular hole, we use {\sf COMSOL} to
  approximate the functions $\mu_i(y)$ from the integral equation
  \eqref{cell_problem1} and to calculate the integral.  This yields 
  precomputed values of the four entries of matrix $\Pi$ that are used
  in the remainder of the computation.
\item[\sf Step 3:] At each time step, advance the solution components as
  follows:
  \begin{description}
  \item[\sf 3a.] Set $T_1=\omega(\HE_1)$ and $\HE_2=\omega^{-1}(T_2)$.
  \item[\sf 3b.] Discretize the microscale heat equation problem
    \eqref{limit_problem_strong2} in space on each local cell
    $\tilde{Y}^2(x,t)$ \emph{within the annular-shaped water region
      only} by using a moving spatial grid with points $y_j= s +
    j(\gamma-s)/\Mmicro$ for $j=0,1, \dots, \Mmicro$ (the ice region has
    known constant temperature $T_1\equiv T_c$ and so no discretization
    is required there).  A finite element approximation is used with
    linear Lagrange elements as basis functions, and the solution is
    integrated over a single time step by freezing the values of
    $\HE_1$, $T_1$ and $s$ at the previous step.  In practice,
    sufficient accuracy is obtained by taking $\Mmicro=4$.
  \item[\sf 3c.] Integrate the macroscale problem
    \eqref{limit_problem_strong1} for $\HE_1$ at the fixed grid points
    $x_i=i \Rtree/\Mmacro$ for $i=0, 1, \dots, \Mmacro-1$. Due to the
    radial symmetry of the domain $Y^2$, the integral in
    (\ref{limit_problem_strong1}a) reduces to $2\pi r D(\HE_2)\nabla
    T_2\cdot \normvec$ where $r$ is the radius of $Y^2$. Again, we use a
    finite element discretization in space with linear Lagrange
    elements, this time freezing $\HE_2$ and $T_2$ at the current
    values.  Note that here is where the precomputed values of
    $\Pi_{ij}$ are required from the reference cell.
  \item[\sf 3d.] Integrate the interface evolution equation
    \eqref{limit_problem_strong3} by using a finite difference
    approximation of the derivative in the right hand side, and by
    freezing values of $T_2$ at the current time. 
  \item[\sf 3e.] Increment the time variable and return to Step 3a. 
  \end{description}
\end{description}

The above algorithm must be modified when the ice completely melts and
the governing equations change in that the interface location $s$
becomes identically zero and the Stefan condition
\eqref{limit_problem_strong3} drops out.  At the same time, the
separation of the reference cell into two sub-domains $Y^1$ and $Y^2$
(which is required to handle the Dirichlet condition on temperature at
the phase interface) is no longer necessary and hence the temperature
can be described by the single field $T_1$ that obeys
\eqref{limit_problem_strong1} alone (with zero right hand side, constant
$D$, and $\Pi\equiv 1$).  This alteration to the governing equations can
be handled easily in the numerical algorithm above by a form of ``event
detection'' based on the ice bar radius $s$: when $s>0$, ice is still
present and the original equations are solved; but as soon as $s=0$, ice
is totally melted and the modified equations just described are solved
instead (omitting steps 3b and 3d).

\subsection{Numerical results}

We now present numerical simulations of the reduced model wherein a
periodic array of ice bars is suspended within a circular water-filled
domain.  The domain has a radius of $\Rtree=0.25\;m$ and the periodic
reference cell $Y=[0,\delta]^2$ depicted in Figure~\ref{fig:ref-cell} is
given a size $\delta=10^{-3}\;m$.  The system is solved in dimensional
variables so that there is no need to non-dimensionalize and hence the
size of the reference cell is simply equal to the physical dimension
$\delta$.  We assume that the ice bars each have initial radius
$s(0)=0.1\,\delta=10^{-4}\;m$ and we take the radius of the artificial
boundary $\Gamma$ in the reference cell to be $\gamma=0.45\,\delta$.  All
other physical parameter values are listed in Table~\ref{tab:params1}.

\begin{table}
  \centering
  \caption{Values of parameters in the reduced model for
    melting ice bars, taken from \cite{AHORN5}.} 
  \label{tab:params1}
  \begin{tabular}{clcc}\hline
    \emph{Symbol} & \emph{Description} & \emph{Value} & \emph{Units}\\\hline
    $c_w$    & Specific heat of water        & $4180$      & $J/kg \, \degK$ \\ 
    $c_i$    & Specific heat of ice          & $2100$      & $J/kg \, \degK$ \\
    $\HE_w$  & Enthalpy of water at $\Tcrit$ & $9.07\times 10^5$ & $J/kg$ \\ 
    $\HE_i$  & Enthalpy of ice at $\Tcrit$   & $5.74\times 10^5$ & $J/kg$ \\ 
    $k_w$    & Thermal conductivity of water & $0.556$     & $W/m\, \degK$ \\
    $k_i$    & Thermal conductivity of ice   & $2.22$      & $W/m\, \degK$ \\
    $\rho_w$ & Density of water              & $1000$      & $kg/m^3$ \\
    $\rho_i$ & Density of ice                & $917$       & $kg/m^3$ \\
    $\Tcrit$ & Melting (critical) temperature& $273.15$    & $\degK$ \\
    $\Tout$  & Ambient (outside) temperature & $\Tcrit+10$ & $\degK$ \\
    $T_{init}$ & Initial temperature         & $\Tcrit$    & $\degK$ \\
    $\delta$ & Length of reference cell      & $1.0\times 10^{-3}$   & $m$\\
    $\gamma$ & Radius of artificial boundary $\Gamma$ & $0.45\,\delta$ & $m$\\
    \hline
  \end{tabular}
\end{table}

Figure~\ref{fig:sims-reduced} presents snapshots at a number of time
points that show the variation in global temperature $T_1$ and the
radius of the ice bars $s$ throughout the domain.  As time goes on, the
temperature gradually increases as heat from the outer boundary diffuses
inwards.  In response to this rise in temperature, the ice melts and the
size of the ice bar in each local cell problem decreases.  The ice in
the outermost regions melts first and by time $t\approx 23 h$, the
entire domain is completely melted.  In contrast with the results in
\cite{AHORN5} that present the progress of thawing at some specific
location on the microscale, these results illustrate the progress of the
thawing front on the macroscale.
\begin{figure}
  \centering
  \subfigure[Temperature.]{\includegraphics[width=0.45\textwidth]{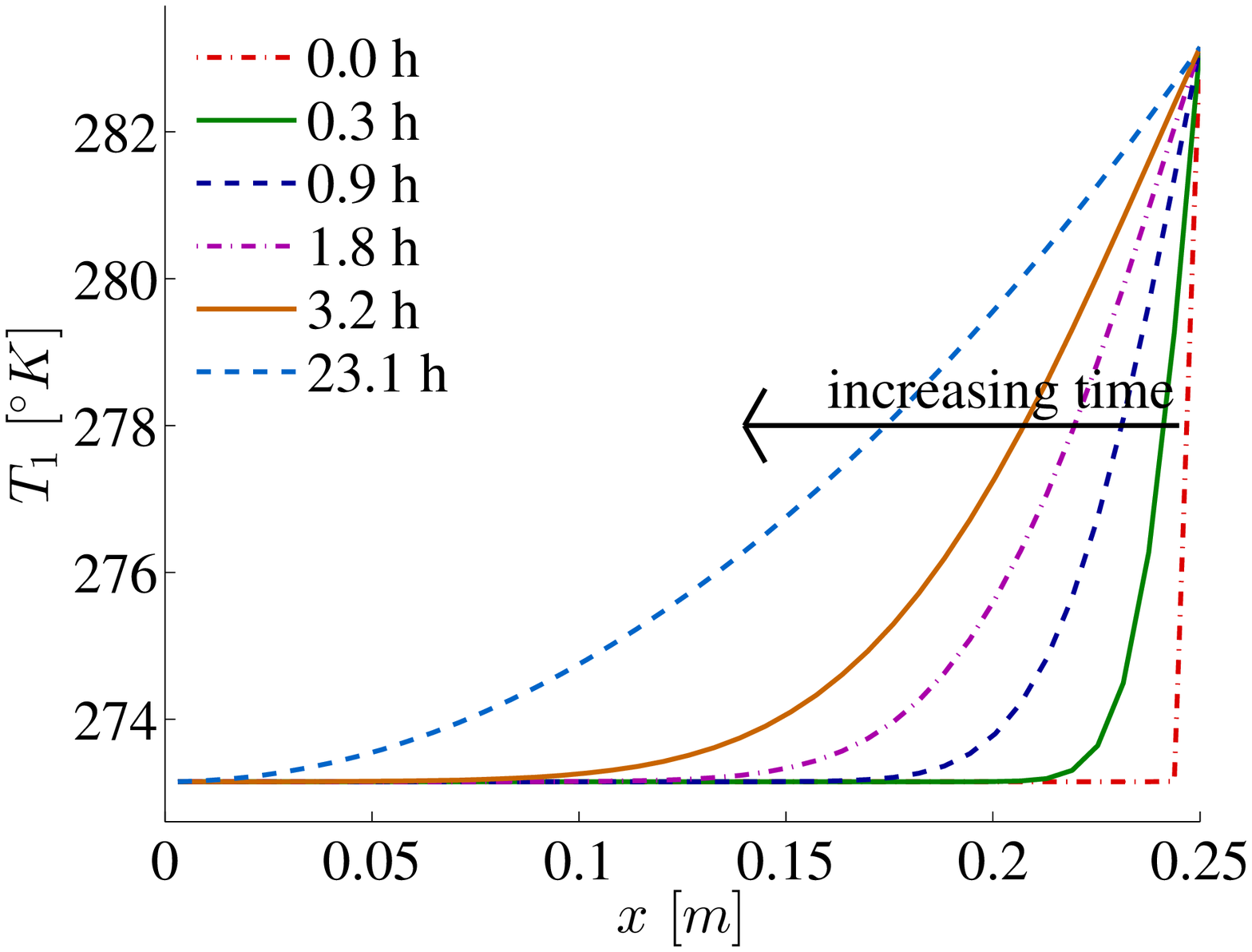}\label{fig:sims-reduced-a}}
  \qquad 
  \subfigure[Local ice bar radius.]{\includegraphics[width=0.45\textwidth]{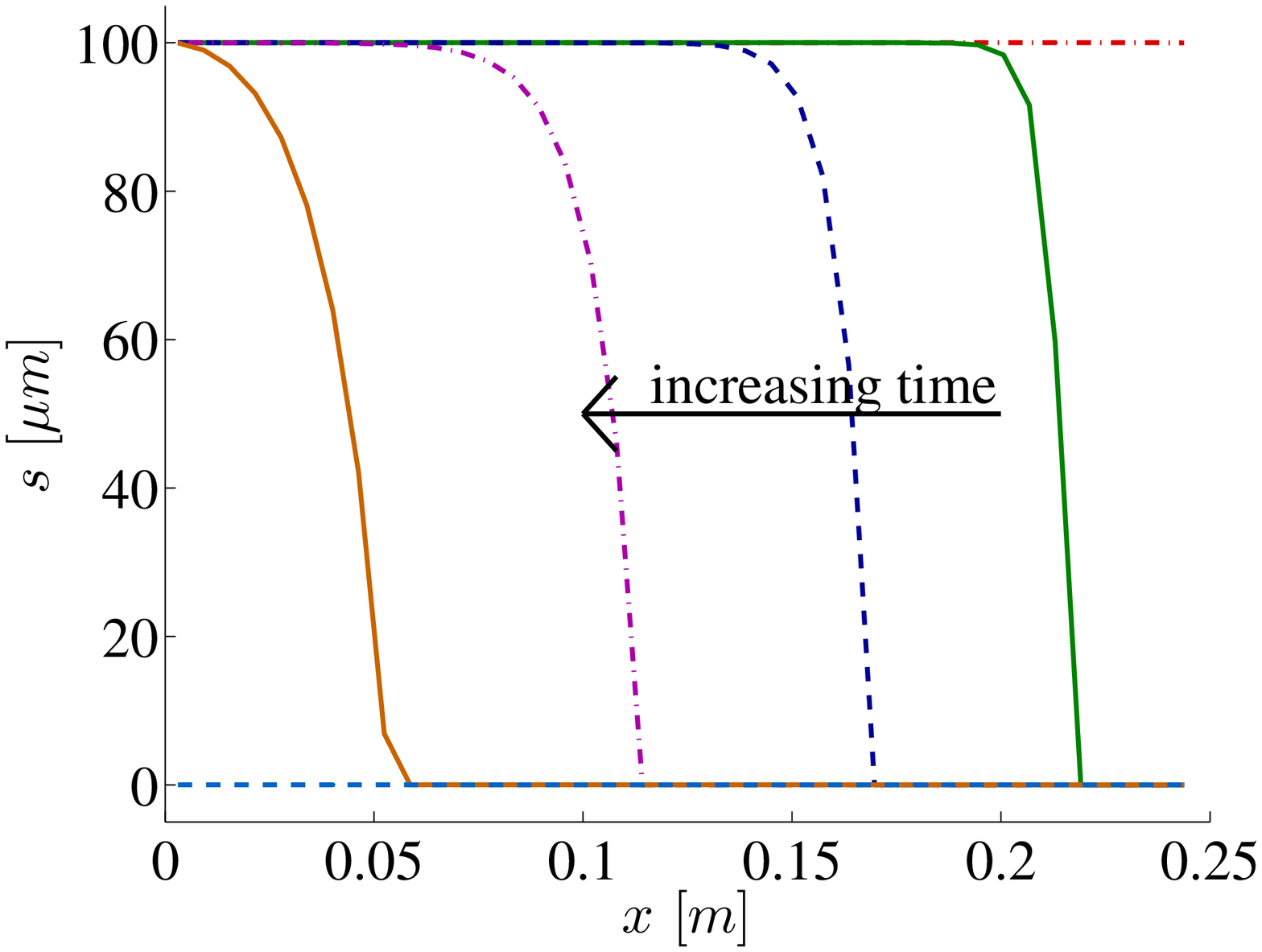}\label{fig:sims-reduced-b}}
  \caption{Solution profiles for the local problem, shown at selected
    times between 0 and 23~h.}
  \label{fig:sims-reduced}
\end{figure}

\section{Application to maple sap exudation}
\label{sec:maple_sap_exudation}

Our aim is now to apply the limit problem \eqref{limit_problem_strong1}
and \eqref{limit_problem_strong2} to the study of sap exudation in maple
trees, which was the original motivation for developing the periodic
homogenization analysis in this paper.  Sugar maple trees (along with a
few other relatives such as black maple, birch, walnut and sycamore)
have a unique ability amongst all tree species in that they exude large
quantities of sap during the winter in the leafless state.  Sap
exudation is driven by an elevated pressure in the tree stem that builds
up over a period of several days in which air temperature oscillates
above and below the freezing point.  The ability of maple to exude sap
has intrigued tree physiologists for nearly a century, and various
physical and biological processes have been proposed to explain this
behaviour~\cite{johnson-tyree-dixon-1987, milburn-kallarackal-1991,
  tyree-1995}.  Although a significant degree of controversy remains
over the root causes of sap exudation the most plausible and
widely-accepted explanation is a freeze/thaw hypothesis proposed by
Milburn and O'Malley \cite{AHORN2}.  The first complete mathematical
model of the physical processes underlying their hypothesis was
developed by Ceseri and Stockie\cite{AHORN5}, who focused on the thawing
half of the cycle.

\subsection{Physical background: The Milburn-O'Malley process}
\label{sec:milburn-omalley}

The Milburn-O'Malley process depends integrally on the distinctive
microstructure of the sapwood (or xylem) in sugar maple trees
(\emph{Acer saccharum}).  Wood in most deciduous tree species consists
of long, straight, roughly cylindrical cells that are on the order of
$1\;mm$ long.  These cells can be classified into two main types:
\emph{vessels} having an average radius of $20\;\mu m$, which are
surrounded by the much more numerous \emph{(libriform) fibers} with a
radius of roughly 3--4\;$\mu m$.  The repeating structure of vessels and
fibers is illustrated in Figure~\ref{fig:xylem}.  The vessels, being
significantly larger, provide the main route for transport of sap from
roots to leaves during the growing season whereas the fibers are
understood to play a largely passive structural role.  Under normal
conditions the vessels are filled with sap, which is mostly made up of
water but can also contain as much as 2--3\%\ sugar by weight for
species like the sugar maple.  On the other hand, fibers are thought to
be primarily filled with gas (i.e., air) and there is often also gas
present within the vessel sap either as bubbles or in dissolved form.
\begin{figure}
  \centering
  \footnotesize\itshape
  \begin{tabular}{ccc}
    \includegraphics[width=0.41\textwidth]{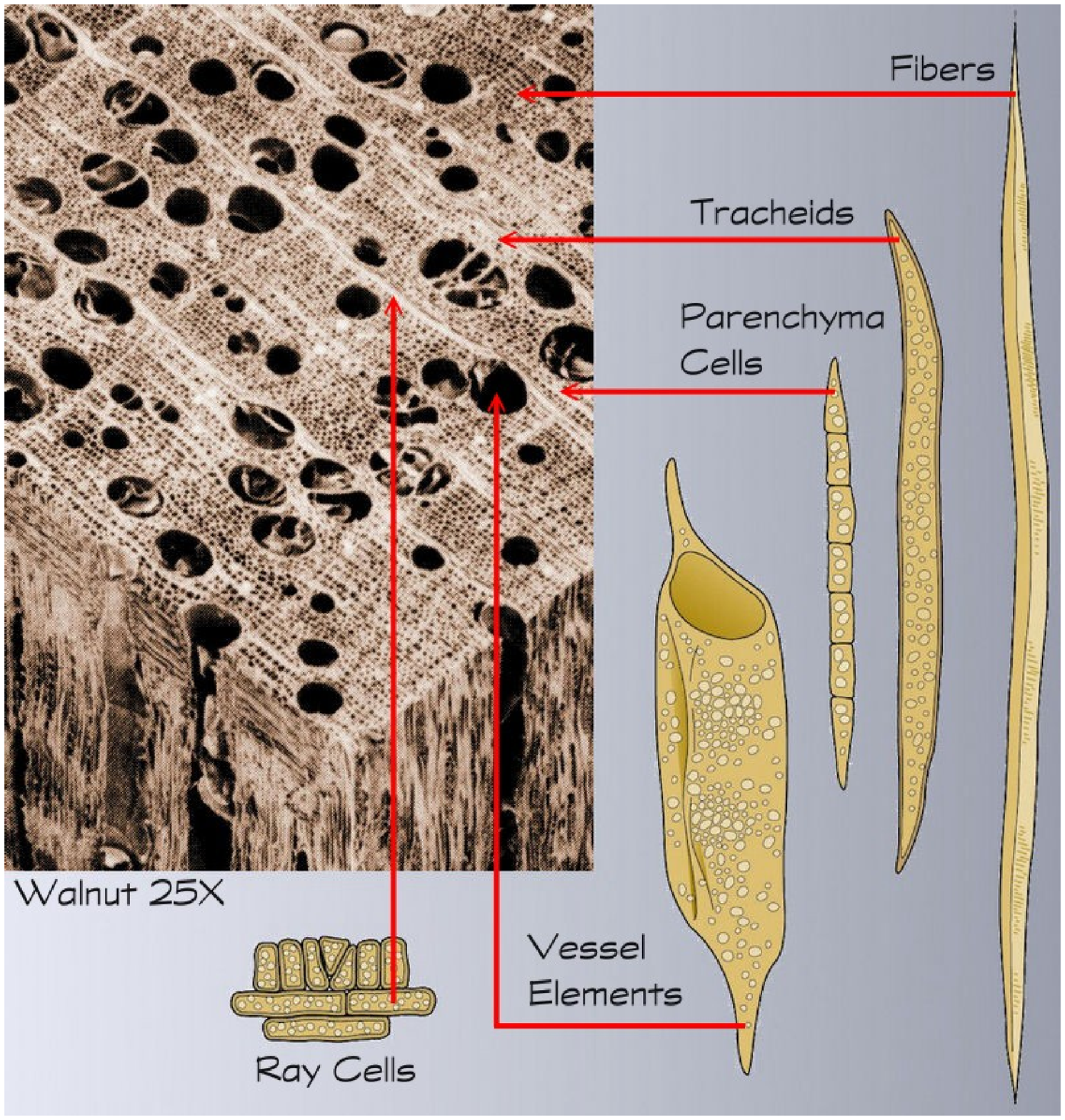}
    & &
    \includegraphics[width=0.52\textwidth]{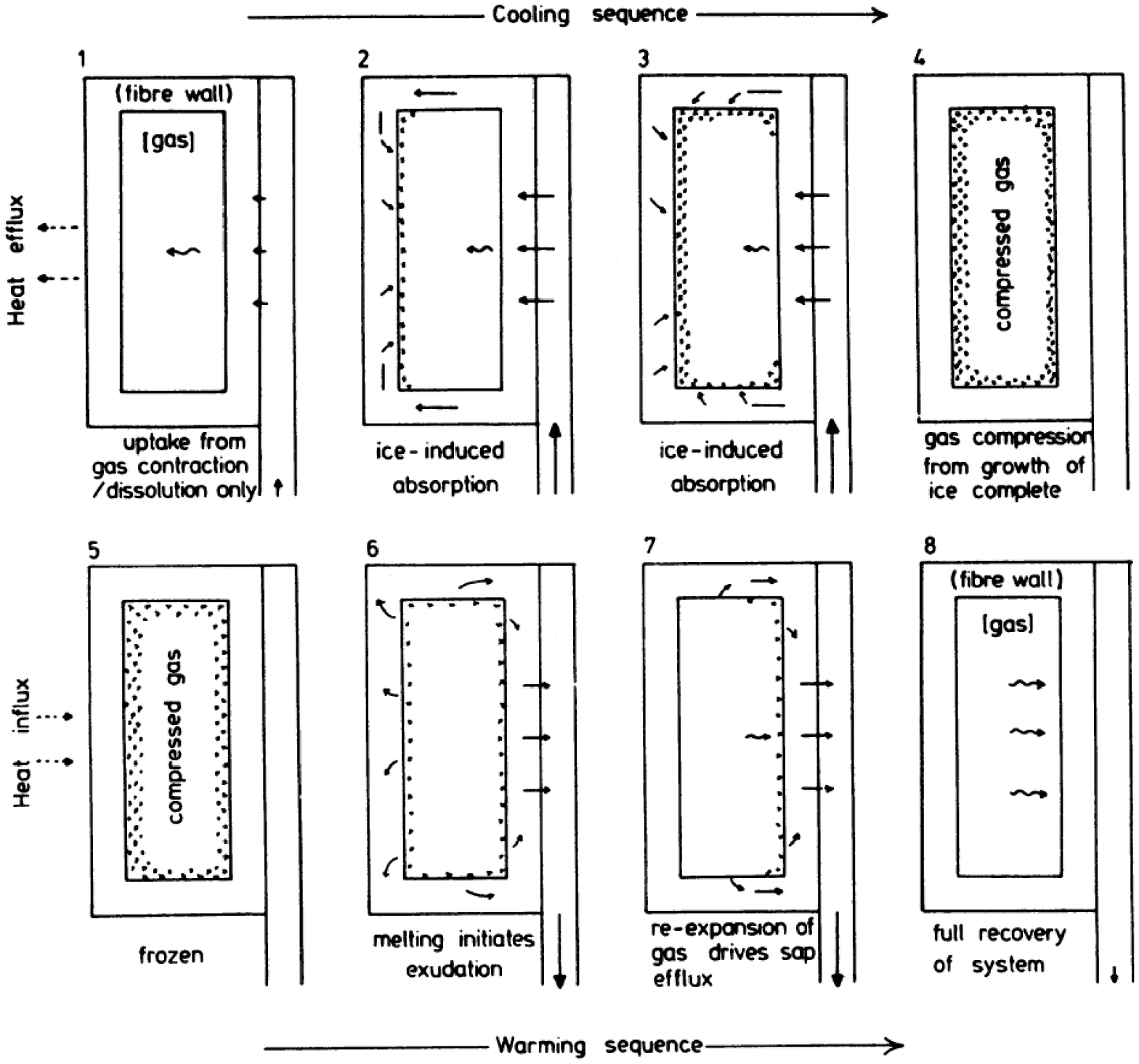}
    \\
    (a) & & (b)
  \end{tabular}
  \caption{(a) A cut-away view of the sapwood (or xylem) in a hardwood
    tree such as sugar maple, illustrating the repeating microstructure
    of vessels surrounded by fibers (the other cell types indicated here
    are ignored in our model).  Source: \emph{Workshop
      Companion}~\cite{workshop-companion}.  
    (b) Illustration of the Milburn-O'Malley
    process taken from~\cite[Fig.~7]{AHORN2} (\copyright\ 2008 Canadian
    Science Publishing, reproduced with permission).  Our focus is the
    ``warming sequence'' in the bottom row, numbers 5--8.  The fiber is
    the large rectangular structure on the left of each image and the
    vessel is represented by a vertical channel on the right (not drawn
    to scale).}
  \label{fig:xylem}
\end{figure}

Milburn and O'Malley hypothesized that during late winter when evening
temperatures begin to drop and just before the onset of freezing, sap is
drawn through tiny pores in the fiber/vessel walls by capillary and
adsorption forces into the gas-filled fibers where it begins to freeze
on the inner surface of the fiber wall (refer to
Figure~\ref{fig:xylem}(b), top ``cooling sequence'').  As temperatures
drop further the ice layer grows and the gas trapped inside is
compressed, forming a pressure reservoir.  When temperatures again rise
above freezing the next day, the process reverses: the ice layer melts
and the pressurized gas drives liquid melt-water back into the vessel
where it then (re-)pressurizes the vessel compartment (the ``warming
sequence'' in Figure~\ref{fig:xylem}(b)).  Milburn and O'Malley also
stressed the importance of incorporating osmotic effects in order to
obtain the high sap pressures actually observed in sugar maple trees,
and this has since been verified in experiments~\cite{AHORN4}.  The
osmotic pressure derives from the selectively permeable nature of the
fiber/vessel wall, which prevents sugars in the vessel sap from entering
the fiber and thereby introduces an osmotic pressure gradient between
the sugar-rich vessel sap and the pure water contained in the fiber. The
presence of sugar also leads to a freezing point depression effect in
which the vessel sap has a lower freezing point than the fiber water.

There are two additional features not included in~\cite{AHORN2} that are
essential in order to obtain physically relevant results.  First of all,
Ceseri and Stockie~\cite{AHORN5} demonstrated the necessity of including
gas bubbles within the vessel sap in order to permit exchange of
pressure between the vessel and fiber compartments.  Secondly, despite
the pervading belief that there is no significant root pressure in maple
during winter~\cite{kramer-boyer-1995,wilmot-2011}, we found to the
contrary that it is essential to include uptake of root water during the
freezing process in order that pressure can accumulate over multiple
freeze/thaw cycles.  Indeed, this second addition is supported by the
recent experiments of Perkins and van~den Berg~\cite{brown-2013} that
demonstrate the existence of significant root pressures during the sap
harvest season.

\subsection{Mathematical formulation}
\label{sec:maple_description}

The physical description of the modified Milburn-O'Malley hypothesis
outlined in the previous section (with the exception of the root
pressure) was used in~\cite{AHORN5} to derive a mathematical model for
the cell-level processes driving sap exudation over a single freezing
cycle.  Our reduced model for melting ice bars described in
Section~\ref{sec:problem1} captures the phase change aspects of this
process, with the ice bars being analogous to regions of frozen water
contained within fiber cells in the sapwood.  However, we need to
add a number of extensions to the reduced model in order to
capture the remainder of the physics inherent in sap exudation:
\begin{enumerate}
\item \emph{Introduce a gas phase:} which takes the form of a gas
  bubble in both fiber and vessel and implies that the diffusion
  coefficient varies in space.  Although we do not provide the analytical
  derivation for this more general class of diffusion coefficient, the
  extension of our analysis is straightforward.
\item \emph{Dissolved sugar in the vessel sap:} which gives rise to an
  osmotic potential between fiber and vessel that is essential for
  generating realistic exudation pressures.  Sugar within the vessel sap
  also depresses the freezing point so that the function $\omega$ in
  \eqref{def_omega} differs between vessel and fiber.  Although we do not
  need to incorporate freezing point depression explicitly in our model,
  it could be handled in several ways, for example by adding an extra
  spatial dependence in $\omega$. Alternatively, the fiber could be
  defined as separate domain that is connected to the vessel via
  appropriate boundary conditions, thereby ensuring that the
  homogenization results carry through for the sap exudation model.  We
  have therefore chosen not to present the analytical derivation for
  this case, although various possible functional forms for $\omega$ can
  be found in \cite{bossavit-damlamian-1981}.
\item \emph{Permeability of the cell wall to water:} which permits the
  fiber/vessel pressure exchange and also sets up the sugar
  concentration difference that drives osmosis.  This in turn introduces
  an additional term in equation \eqref{limit_problem_strong3} (for the
  strong formulation) that governs the dynamics of the ice/water
  interface.
\end{enumerate}
Although we do take the necessary step of including the gas phase in
both cell chambers (fiber and vessel) we assume for the sake of
simplicity that the effects of gas dissolution and nucleation are
negligible.  This is the primary difference between our microscale model
and the model in~\cite{AHORN5} on which it is based, although
re-introducing gas dissolution effects would certainly be an interesting
topic for future research.

With the above modifications in mind, we consider the modified reference
cell geometry pictured in Figure~\ref{fig:sap-geometry}a wherein a
circular fiber of radius $\Rf$ is located adjacent to a vessel
compartment.  The fiber is sub-divided into nested annular regions
containing gas, ice and liquid, and the outer radii of the phase
interfaces are denoted $s_{gi}$ (for gas/ice) and $s_{iw}$ (for ice/water).
The vessel also contains a circular gas bubble of radius $r$.  One
additional time-dependent variable is introduced to measure the total
volume of water transferred from fiber to vessel, denoted $U$.  These
four variables depend on time $t$ and also on the position $x$ of the
reference cell in space.  The region lying outside the fiber and inside
the boundary of the square reference cell represents the sap-filled
vessel.  Although the shape of this vessel region clearly does not
correspond to the cylindrical shape of real vessels, we argue that it is
nonetheless a reasonable approximation considering that the vessel is so
much larger than the fiber.  This reference cell geometry should be
contrasted with that used in \cite[Fig.~3.1]{AHORN5}.
\begin{figure}
  \centering
  \begin{tabular}{ccc}
    (a) Reference cell, with ice.\mbox{\hspace*{2.5cm}}
    &&
    (b) Reference cell, without ice.\\
    \includegraphics[height=0.28\textheight]{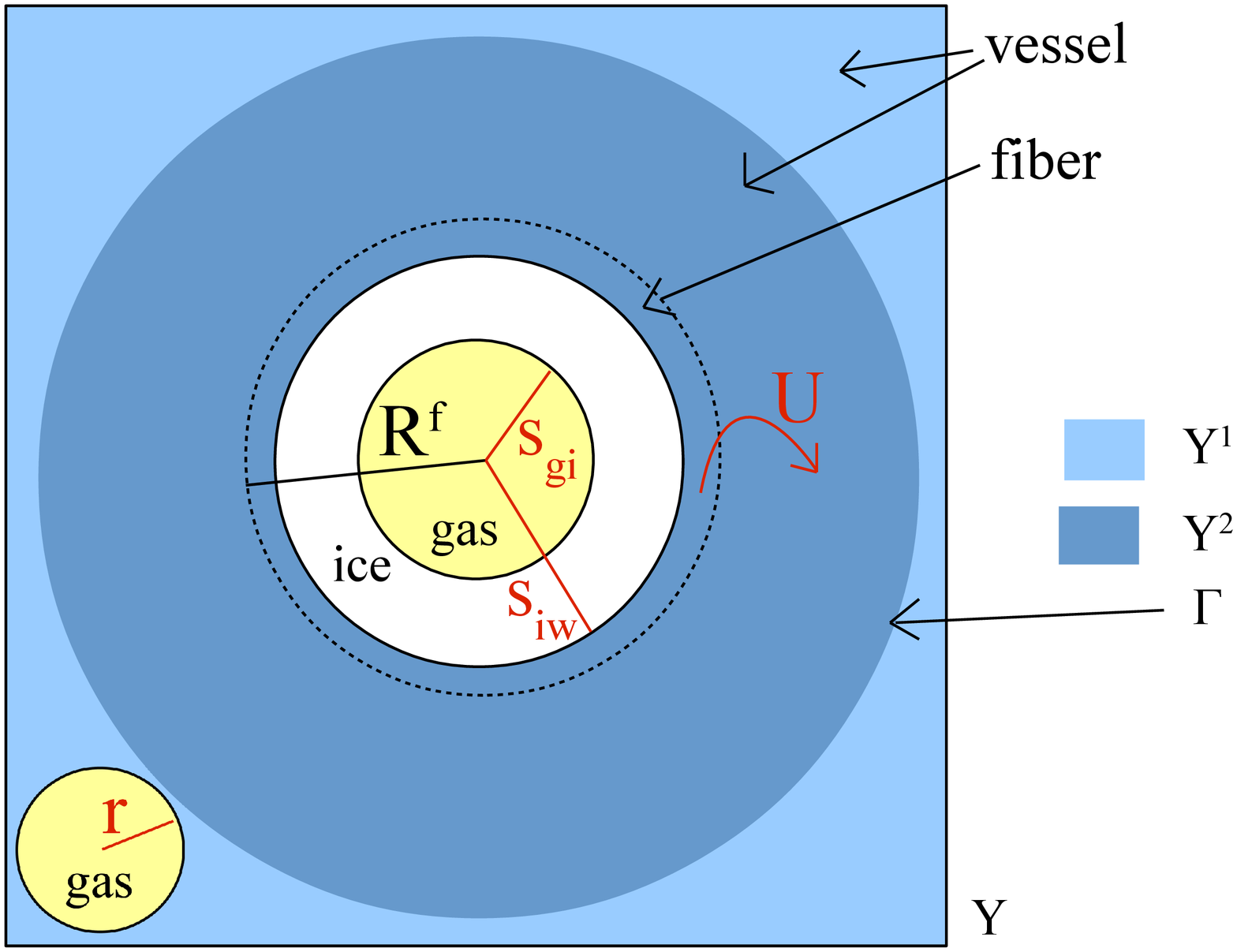}
    && 
    \includegraphics[height=0.28\textheight]{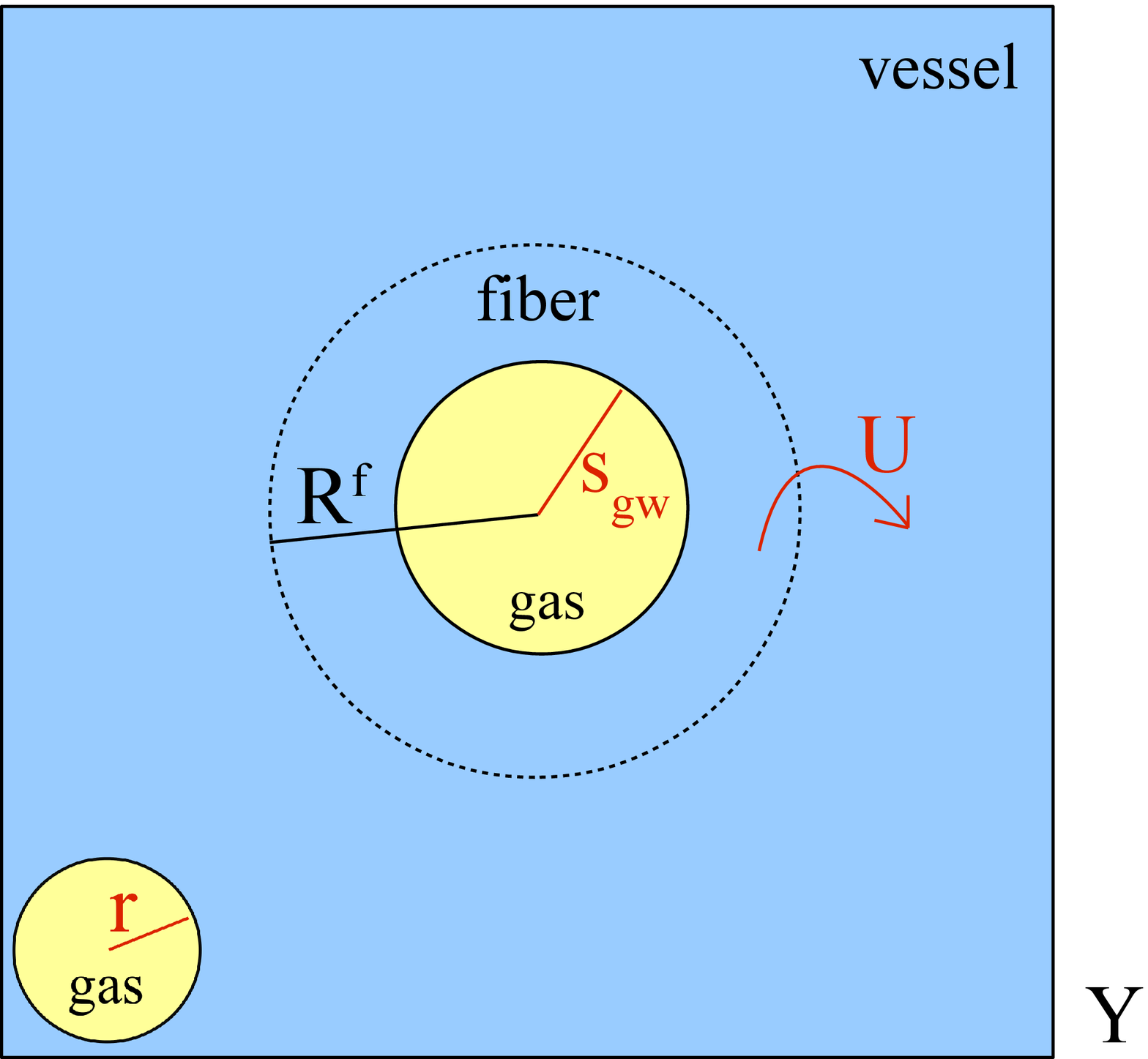}
  \end{tabular}
  \caption{Geometry of the reference cell for the maple sap exudation
    problem.  (a) Initially, the vessel compartment contains a gas
    bubble with radius $r$, while the circular fiber (radius $\Rf$)
    contains a gas bubble (radius $s_{gi}$), surrounded by an ice layer
    (thickness $s_{iw}-s_{gi}$), and finally a layer of melt-water
    (thickness $\Rf-s_{iw}$).  The porous wall between fiber and vessel
    is denoted by a dotted line.  As ice melts, the melt-water is forced
    out by gas pressure through the porous fiber wall into the
    surrounding vessel compartment.  The total volume of melt-water
    transferred from fiber to vessel is denoted $U$.  (b) After the ice
    has completely melted, the problem simplifies significantly leaving
    only a single interface ($s_{gw}$, between gas/water) and
    a single temperature field $T_1$ for the domain $Y\equiv Y^1$.}
  \label{fig:sap-geometry}
\end{figure}

We are now prepared to describe the governing equations for the
homogenized sap exudation problem.  The domain of interest $\Omega$ is a
2D circular cross-section of a tree stem.  At every point in $\Omega$ is
associated a reference cell $Y$, and hence the actual domain of interest
consists of $\Omega\times Y$.  In comparison with the reduced problem
studied in Section~\ref{sec:problem1}, the sap exudation problem
includes additional phases that need to be incorporated into the
local/global temperatures as well as the phase interfaces, as described
under points 1--3 above.  The equations for temperature are similar to
those for the reduced model so that the periodic homogenization results
require no significant changes.  Therefore, the same limit problem
\eqref{limit_problem_strong1} and \eqref{limit_problem_strong2} derived
earlier for the reduced model may also be used for the maple sap
exudation application, with Eq.~\eqref{limit_problem_strong3} replaced
by Eq.~\eqref{cell_problem_maple_siw}.

For the reference cell problem, the dynamics for $s_{iw}$, $s_{gi}$, $r$
and $U$ are governed by four differential equations whose
derivation can be found in \cite{AHORN5}:
\begin{subequations}
  \label{cell_problem_maple}
  \begin{align}
    \partial_t s_{iw} &= -\frac{D(\HE_2)}{\HE_w-\HE_i}\nabla_y T_2\cdot
    \normvec + \frac{1}{2\pi s_{iw} \Lf} \, \partial_t U, 
    & \quad \text{(Stefan condition)}
    \label{cell_problem_maple_siw}
    \\ 
    0 &= -\rho_is_{gi}\partial_t s_{gi} +
    (\rho_i-\rho_w)s_{iw}\partial_t s_{iw} + \frac{1}{2\pi \Lf}\, \partial_t U, 
    & \quad \text{(conservation of mass)}
    \label{cell_problem_maple_sgi}
    \\ 
    \partial_t r &= \frac{N}{2\pi r \Lv}\,  \partial_t U,
    & \quad\text{(conservation of volume)}
    \label{cell_problem_maple_r}
    \\ 
    \partial_t U &= -\frac{KA}{N \rho_wgW}(p_w^v - p_w^f - \gasR C_s T_1).
    &\quad\text{(Darcy's law and osmotic pressure)}
    \label{cell_problem_maple_U}
  \end{align}
\end{subequations}
Here we have introduced new variables for pressure $p$ and density
$\rho$, and the superscripts $f$/$v$ refer to fiber/vessel while
subscripts $g$/$i$/$w$ refer to gas/ice/water phases.  We make
particular note of the fact that in sapwood there are many more fibers
than vessels (see Figure~\ref{fig:xylem}a) and so the effect of any
fiber-vessel flux terms must be increased to take into account the
multiplicity of the fibers.  With this in mind, we have introduced the
parameter $N$ in
\eqref{cell_problem_maple_r}--\eqref{cell_problem_maple_U} that
represents the average number of fibers per vessel, and for which we
choose $N=16$.  Note also that in the Stefan
condition~\eqref{cell_problem_maple_siw} for the motion of the ice/water
interface we have incorporated an additional term ${\partial_t U}/({2\pi
  s_{iw} \Lf})$ that was neglected in~\cite{AHORN5} and serves to
capture the change in fiber water volume due to outflow through the
fiber/vessel wall.

Also appearing in \eqref{cell_problem_maple} are several intermediate
variables that are given by the following algebraic relations
\begin{subequations}
  \label{algebraic}
  \begin{align}
    p_w^f &= p_g^f(0)\frac{s_{gi}(0)^2}{s_{gi}^2} -
    \frac{\sigma_w}{s_{gi}} , 
    & \quad\text{(Young-Laplace equation for fiber)}
    \label{algebraic_pwf}
    \\
    p_w^v &= p_g^v - \frac{\sigma_w}{r},  
    & \quad\text{(Young-Laplace equation for vessel)}
    \label{algebraic_pwv}
    \\
    p_g^v &= \frac{\rho_g^v \gasR T_1}{M_g} , 
    & \quad\text{(ideal gas law for vessel)}
    \label{algebraic_pgv}
    \\
    \rho_g^v &= \frac{\rho_g^v(0)\Vv_g(0)}{\Vv_g + \Henry(\Vv - \Vv_g)}, 
    \qquad\qquad\qquad\qquad\qquad
    & \quad\text{(density-volume correlation)}
    \label{algebraic_rhogv}
    \\
    V_g^v &= \pi r^2 \Lv, 
    & \quad\text{(volume of gas in vessel)}
    \label{algebraic_Vgv}
  \end{align}
\end{subequations}
where $V$ represents volume.  All other quantities appearing in the
above equations that have not been introduced previously correspond to
constant parameters whose definitions and numerical values are listed in
Table~\ref{tab:params2}.

In any given reference cell, a thawing scenario will eventually reach a
time when the fiber ice layer is completely melted, for which the
reference cell is pictured in Figure~\ref{fig:xylem}b.  In the absence
of ice, the macroscale equations remain unchanged but the microscale
problem must be modified in a manner very similar to what we did for the
reduced problem at the end of Section~\ref{sec:sims-algo}.  In the sap
exudation case the diffusion coefficient $D$ must account for the fact
that there are two possible values, one in the gas region and one in the
water-filled region.  Furthermore, the microscale equation for the
ice/water interface \eqref{cell_problem_maple_siw} drops out and we
identify $s_{iw} \equiv s_{gi} \eqdef s_{gw}$, where the latter refers
to the new fiber gas/water interface.  This leads to a slight
simplification in \eqref{cell_problem_maple_sgi} since the two
time-derivative terms involving $\rho_i$ cancel.  Otherwise, the
microscale equations \eqref{cell_problem_maple} and \eqref{algebraic}
remain the same.


\subsection{Simulations of maple sap exudation}
\label{sec:maple_simul}

A tree stem is roughly a right circular cylinder, and so we employ in
our numerical simulations a macroscopic domain $\Omega$ consisting of a
2D horizontal stem cross-section in the shape of a circle of radius
$\Rtree$.  Because of radial symmetry, we therefore have a
one-dimensional problem that varies only in the radial direction.  We
take a representative value of $\Rtree=0.25\;m$ and choose the length of
the reference cell to be $\delta=3.60\times 10^{-5}\; m$.  The initial
temperature of the tree is taken to be $T_{init}=0\degC$ throughout,
with the fiber water initially frozen and the vessel water in liquid
form owing to the freezing point depression effect.  There is no need to
explicitly incorporate this effect into the model since we are only
concerned with the thawing process, where the initial temperature is at
the melting point of pure water; hence, the only impact of freezing
point depression in this paper is on the initial conditions in that the
vessel water is initially in liquid form.  The outer boundary of the
tree stem is held equal to an ambient temperature $10\degC$ above the
freezing point, corresponding to $T_1(0.25,t)=\Tout=\Tcrit+10$, while
the symmetry condition $\partial_x T_1(0,t)=0$ is imposed at the center
of the domain.  All parameters that are common to the reduced model are
listed in Table~\ref{tab:params1} and their values remain the same,
whereas any new parameters introduced in the sap exudation model are
listed in Table~\ref{tab:params2}.

We remark that a Robin boundary condition (or convective heat transport
condition) is a more appropriate condition to impose at the outer tree
surface.  However, we have chosen to use a Dirichlet boundary condition
here because it is consistent with the homogenized problem derived in
Section~\ref{sec:problem1}.  The extension to the Robin condition is
actually quite straightforward to implement numerically, but it is a
major extension to our analysis and so we relegate it to future work.

\begin{table}
  \centering
  \caption{Constant parameter values and initial conditions used in the
    sap exudation model, with most values taken from \cite{AHORN5}.} 
  \label{tab:params2}
  \begin{tabular}{clcc}\hline
    \emph{Symbol} & \emph{Description} & \emph{Value} & \emph{Units}\\\hline
    $\delta$   & Length of reference cell & $3.60\times 10^{-5}$& $m$\\
    $\Rf$      & Fiber radius           & $3.5\times 10^{-6}$  & $m$\\
    $\Lv$      & Vessel length          & $5.0\times 10^{-4}$  & $m$\\
    $\Lf$      & Fiber length           & $1.0\times 10^{-3}$  & $m$\\  
    $\Vf$      & Fiber volume $=\pi({\Rf})^2 \Lf$  & $3.85\times 10^{-14}$ & $m^3$\\
    $\Vv$      & Vessel volume $=\delta^2 \Lv-\Vf$ & $6.10\times 10^{-13}$ & $m^3$\\
    $A$        & Fiber surface area $=2\pi \Rf \Lf$& $2.20\times 10^{-8}$  & $m^2$\\
    $W$        & Thickness of fiber/vessel wall & $3.64\times 10^{-6}$ & $m$\\
    $N$        & Number of fibers per vessel & 16              & \\
    $\Rtree$   & Tree radius                 & 0.25            & $m$ \\
    \hline
    $g$        & Gravitational acceleration  & 9.81            & $m/s^2$\\
    $\Henry$   & Henry's constant            & 0.0274          & ---\\
    $M_g$      & Molar mass of air           & 0.029           & $kg/mol$\\
    $\gasR$    & Gas constant                & 8.314           & $J/mol\, \degK$\\
    $\sigma_w$ & Water surface tension       & 0.076           & $N/m$\\
    $C_s$      & Vessel sugar concentration & 58.4  & $mol/m^3$\\
    $K$        & Wall hydraulic conductivity& $1.98\times 10^{-14}$& $m/s$\\
    $\Tcrit$   & Melting temperature              & 273.15 & $\degK$\\
    $\Tout$    & Ambient temperature $=\Tcrit+10$ & 283.15 & $\degK$\\
    \hline
    $T_{init}$ & $=\Tcrit$           & 273.15 & $\degK$\\
    $s_{iw}(0)$& $= \Rf$             & $3.5\times 10^{-6}$ & $m$\\
    $s_{gi}(0)$& $= {\Rf}/{\sqrt2}$  & $2.5\times 10^{-6}$ & $m$\\
    $r(0)$     &                     & $6.0\times 10^{-6}$ & $m$\\
    $U(0)$     & & 0                 & $m^3$\\
    $p^f_g(0)$ & & $2.0\times 10^5$  & $N/m^2$\\
    $p^v_g(0)$ & & $1.0\times 10^5$  & $N/m^2$\\
    $p^f_w(0)$ & & $9.89\times 10^4$ & $N/m^2$\\
    $p^v_w(0)$ & & $9.95\times 10^4$ & $N/m^2$\\
    \hline
  \end{tabular}
\end{table}

We now give a brief description of our numerical solution algorithm,
which solves the coupled system of governing equations using a split-step
approach that alternates between solving the microscale (reference cell)
and macroscale solutions.  First, the microscale problem consisting of
the differential-algebraic system \eqref{limit_problem_strong2} and
\eqref{cell_problem_maple} is solved using a finite element
discretization analogous to what we employed for the reduced problem in
Section~\ref{sec:sims-algo}. The temperature equation is
semi-discretized in space using linear Lagrange elements on an equally
spaced radial grid, and when coupled with \eqref{cell_problem_maple} the
resulting system of ODEs is integrated in time using the {\sf Matlab}
solver {\tt ode15s}.  In the second step, the macroscale equation
\eqref{limit_problem_strong1} is discretized similarly and the global
temperature is updated using values of the microscale variables just
computed.

After applying the method just described along with the initial and
boundary conditions stated earlier, the resulting solutions for $T_1$,
$s_{iw}$ and $s_{gi}$, $r$, $U$, $p_w^f$ and $p_w^v$ are illustrated in
Figure~\ref{fig:sap1} at a sequence of six times between 0 and
$2\;h$. In each plot, the horizontal ($x$) axis measures radial distance
from the center of the tree cross-section.
\begin{figure}
  \centering
  \footnotesize\itshape
  \begin{tabular}{ccc}
    (a) Temperature & \qquad & (b) Phase interfaces\\
    \includegraphics[width=0.40\textwidth]{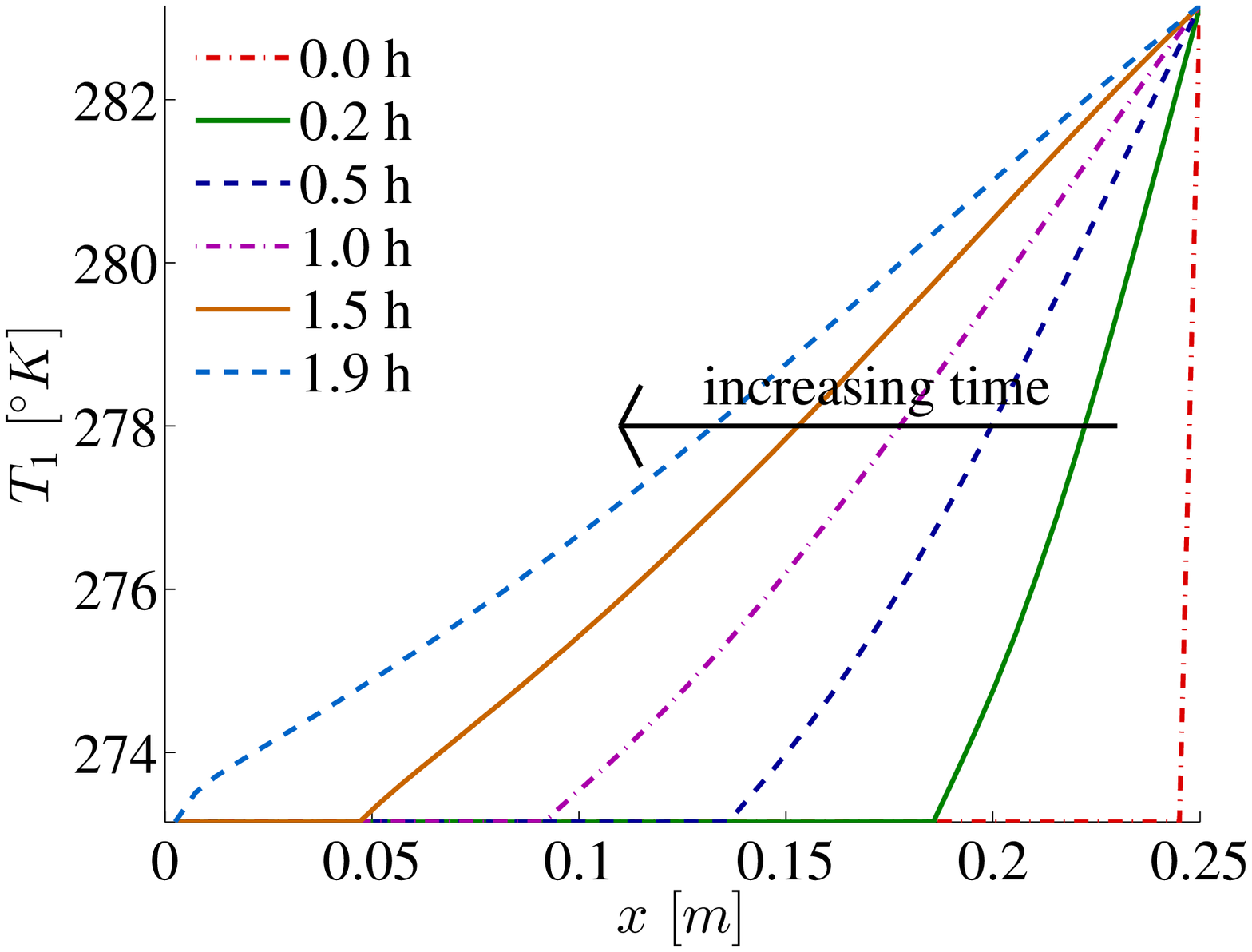}
    & &
    \includegraphics[width=0.40\textwidth]{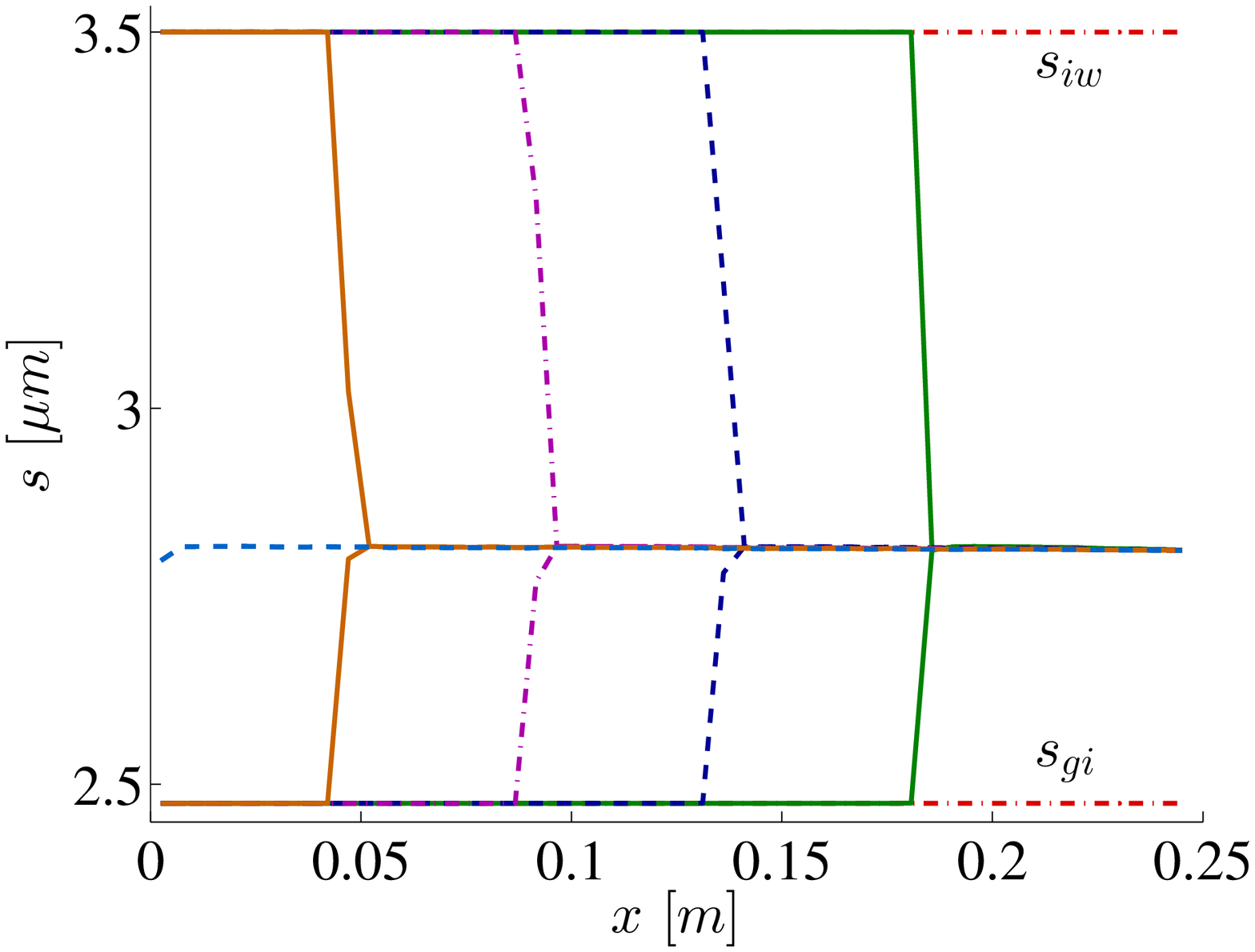}
    \\
    (c) Melt-water volume & & (d) Vessel bubble radius\\
    \includegraphics[width=0.40\textwidth]{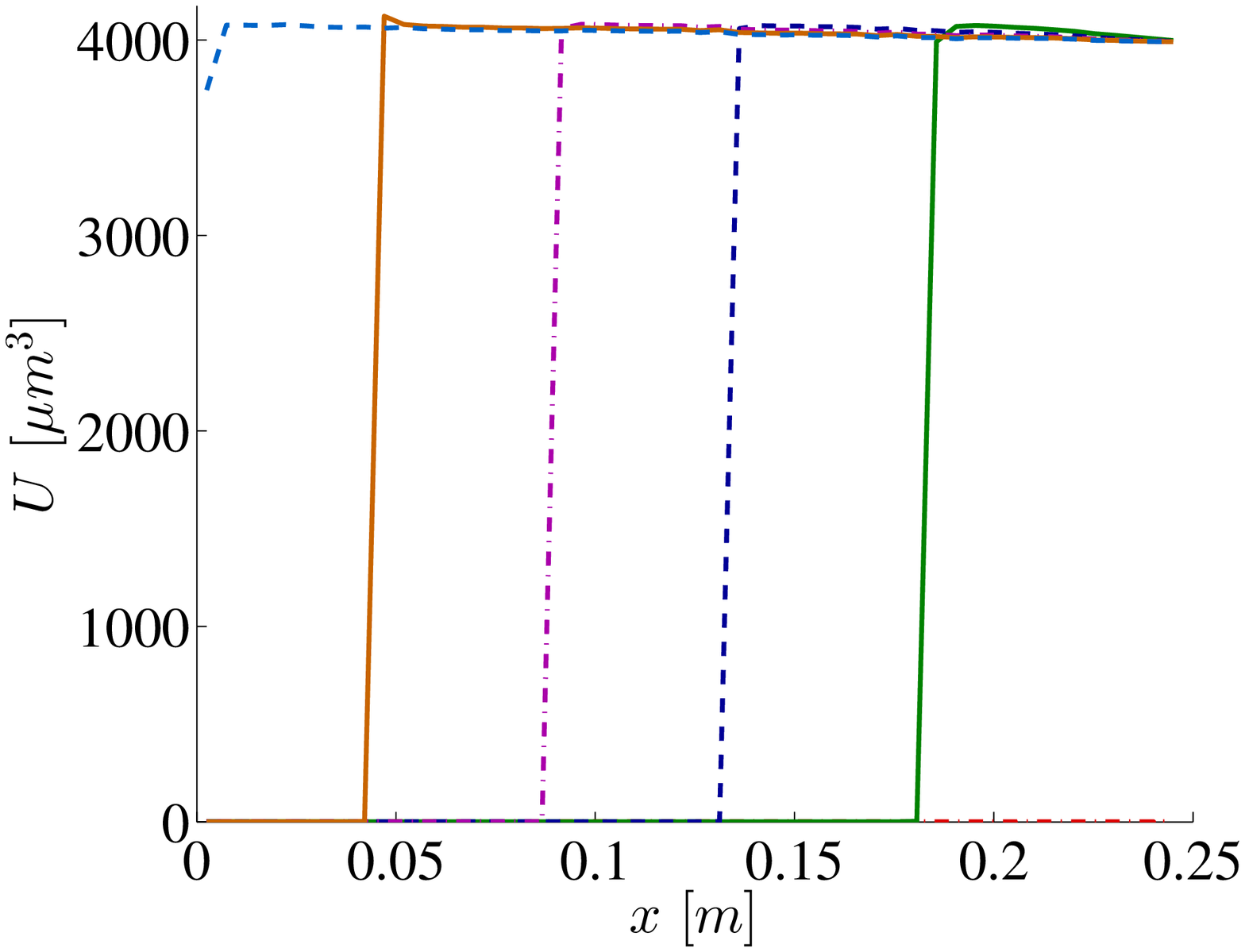}
    & &
    \includegraphics[width=0.40\textwidth]{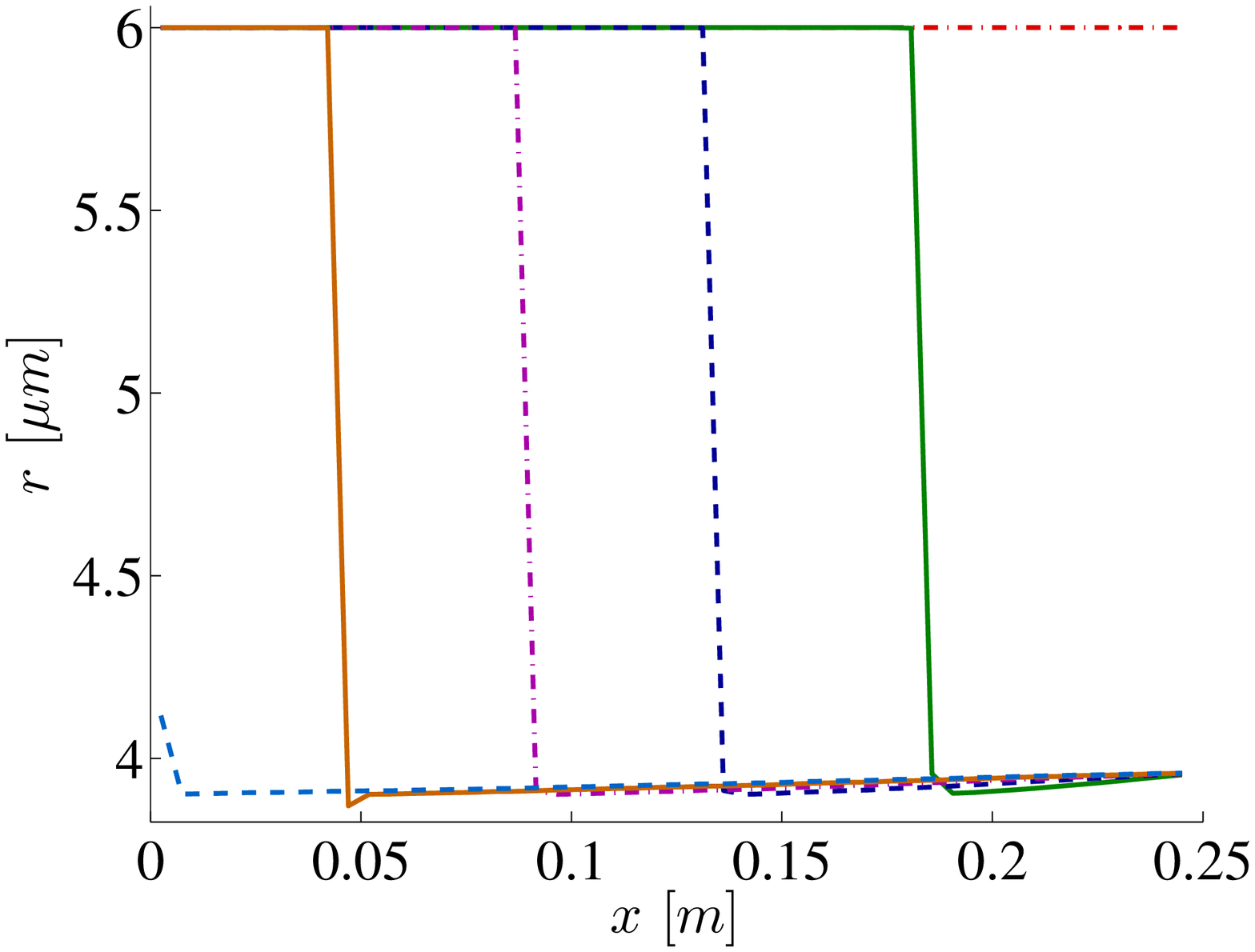}
    \\
    (e) Fiber water pressure & & (f) Vessel sap pressure\\
    \includegraphics[width=0.40\textwidth]{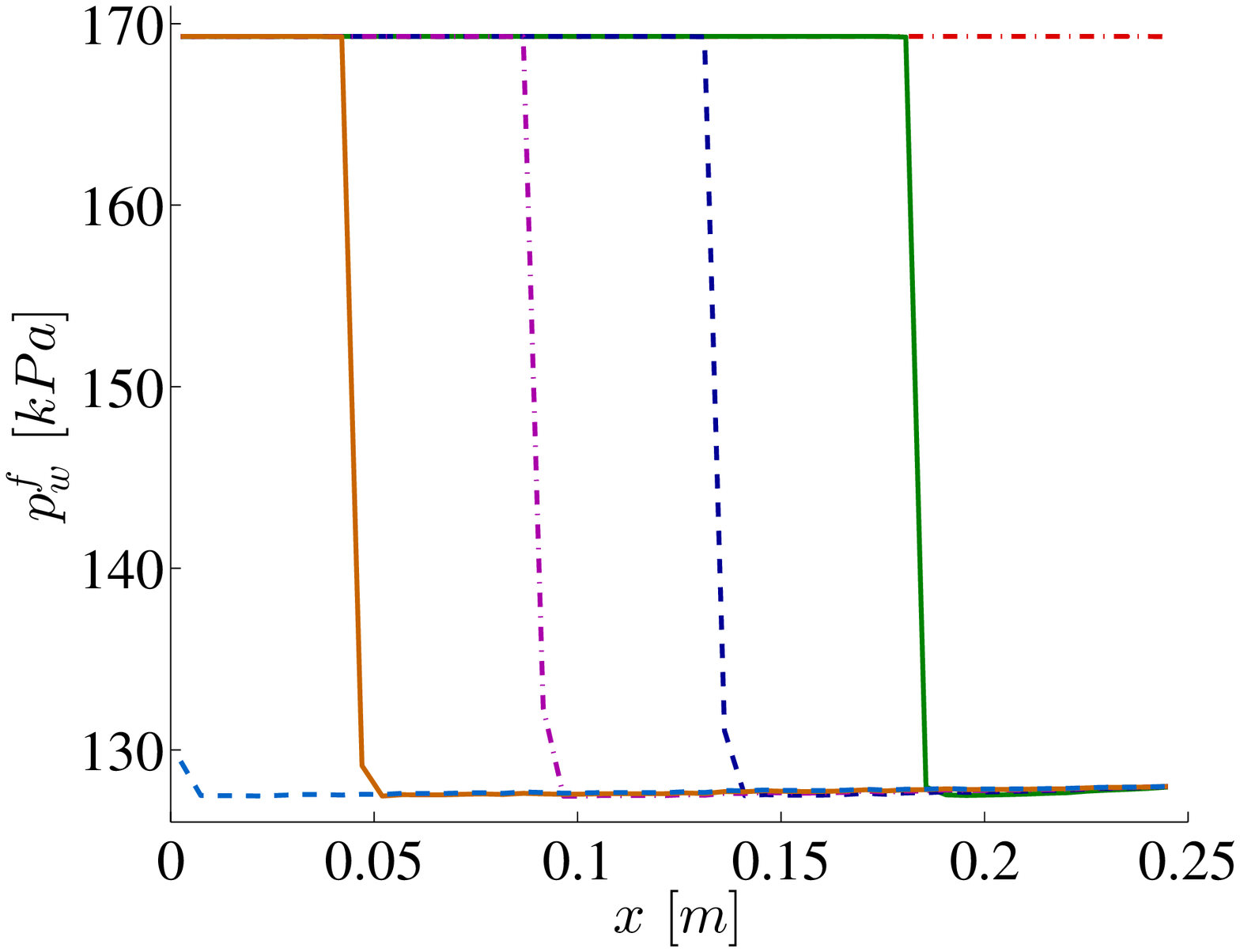}
    & &
    \includegraphics[width=0.40\textwidth]{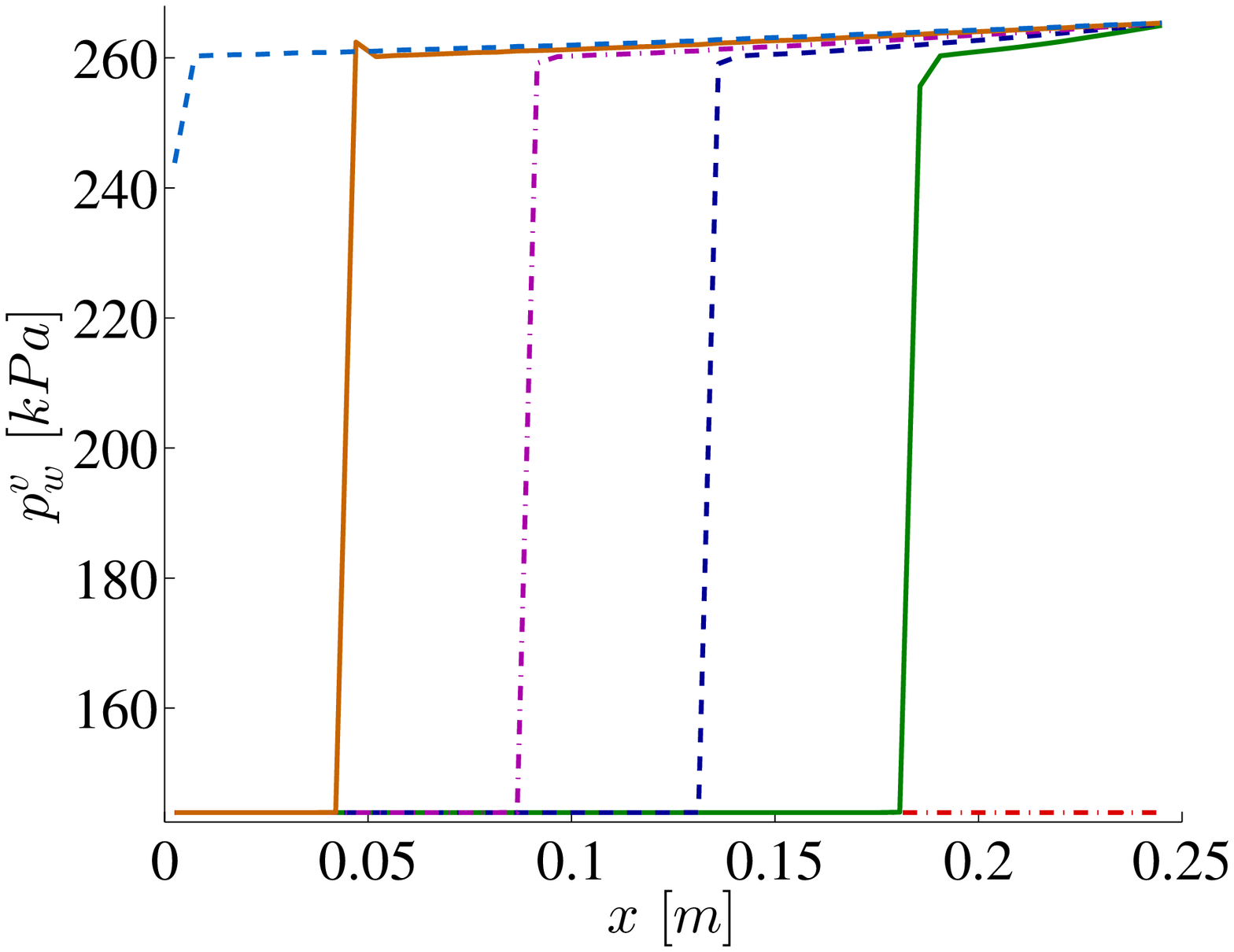}
  \end{tabular}
  \caption{Sap exudation simulations showing the solution profiles at a 
    sequence of time points.  In all cases, the profiles evolve from
    right to left as indicated by the arrow in (a).}
  \label{fig:sap1}
\end{figure}
From these plots, it is clear that the solution dynamics for all
variables appear as a \emph{melting front} that progresses through the
tree from outside to inside (from right to left in the plots) as the
warm ambient air gradually heats up the interior of the tree.
Furthermore, the time scale for complete melting of the ice contained in
the fibers is just under 2\;$h$.

The temperature profiles vary smoothly in space while other solution
quantities are characterized by a steep front that propagates toward the
centre of the tree at a speed that decreases with time.  The steepness
of the melting front derives from the rapid thawing of ice and
subsequent adjustment of sap between vessels and fibers on the
microscale, all of which occurs during the instant after the temperature
exceeds the melting point $\Tcrit$ at any given location.  The reason
for the gradual slowing of the melting front with time is that the heat
flux naturally decreases the closer the front is to the center of the
tree, which in turn leads to a speed decrease owing to the Stefan
condition.

There exists a wide separation in the time scales between the slow
evolution of the global temperature and the relatively rapid local phase
change and sap redistribution, which is clear from the plots in
Figure~\ref{fig:sap2} that show the time variation of the solution at
radial location $x=0.15\;m$.  These plots are consistent qualitatively
with results from simulations of the microscale sap thawing model
presented in \cite{AHORN5}.  We observe that the thickness of the fiber
ice layer, which corresponds to the vertical distance between the
$s_{iw}$ and $s_{gi}$ curves in Figure~\ref{fig:sap2}(b), rapidly drops
to zero as the ice melts.  At the same time, melt-water is driven from
fiber to vessel by the pressure stored in the fiber gas bubble, and the
pressure plot in Figure~\ref{fig:sap2}(c) clearly illustrates the
subsequent increase in $p^v_w$ that we attribute to {exudation
  pressure}.  After the melting process is complete, the vessel water
pressure continues to increase (at a much slower rate not easily visible
to the naked eye) owing to further increases in temperature and
consequent expansion of the gas in both fiber and vessel.
\begin{figure}
  \centering
  \footnotesize\itshape
  \begin{tabular}{ccc}
    (a) Temperature & (b) Phase interfaces & (c) Liquid pressures \\
    \includegraphics[width=0.30\textwidth]{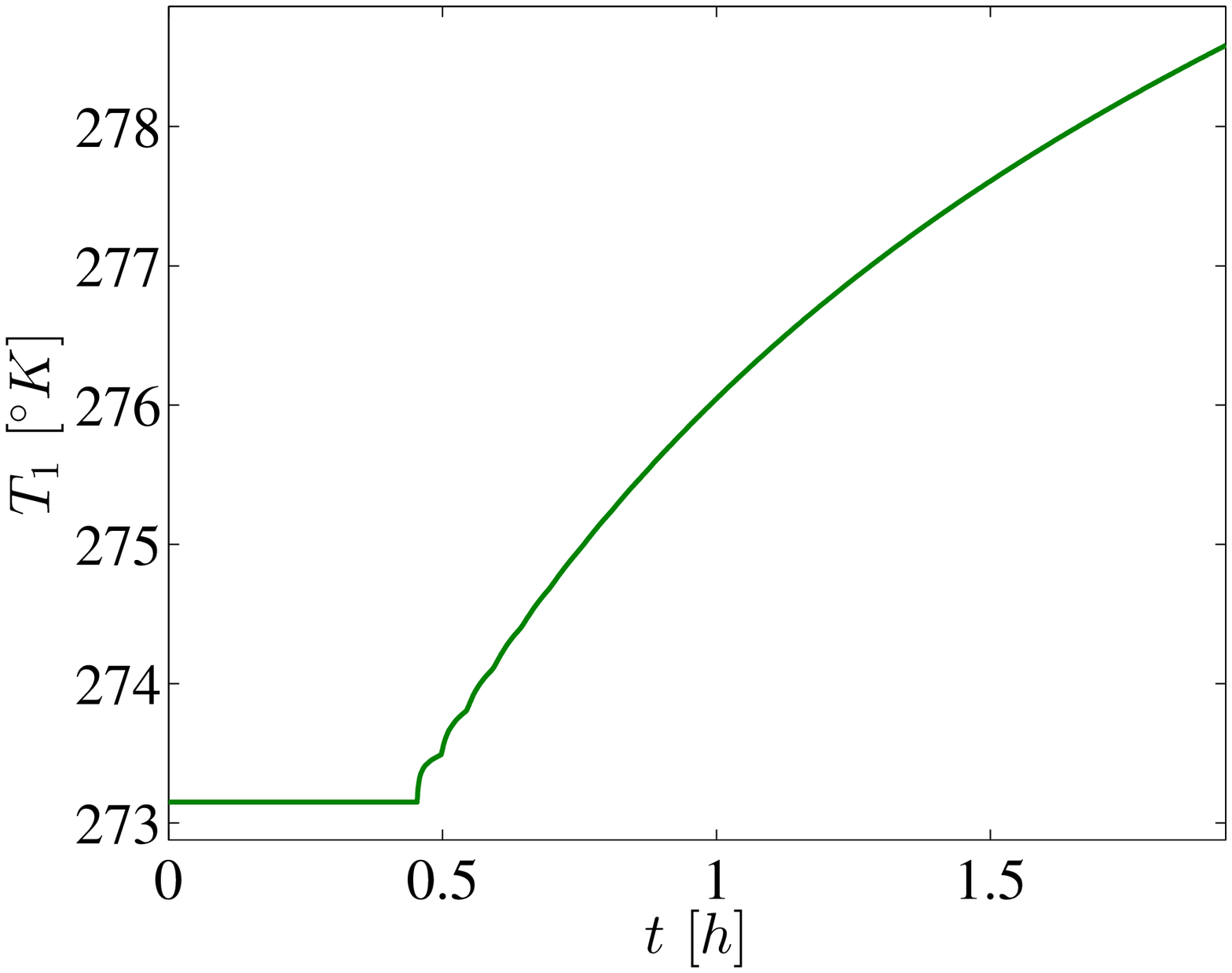} & 
    \includegraphics[width=0.30\textwidth]{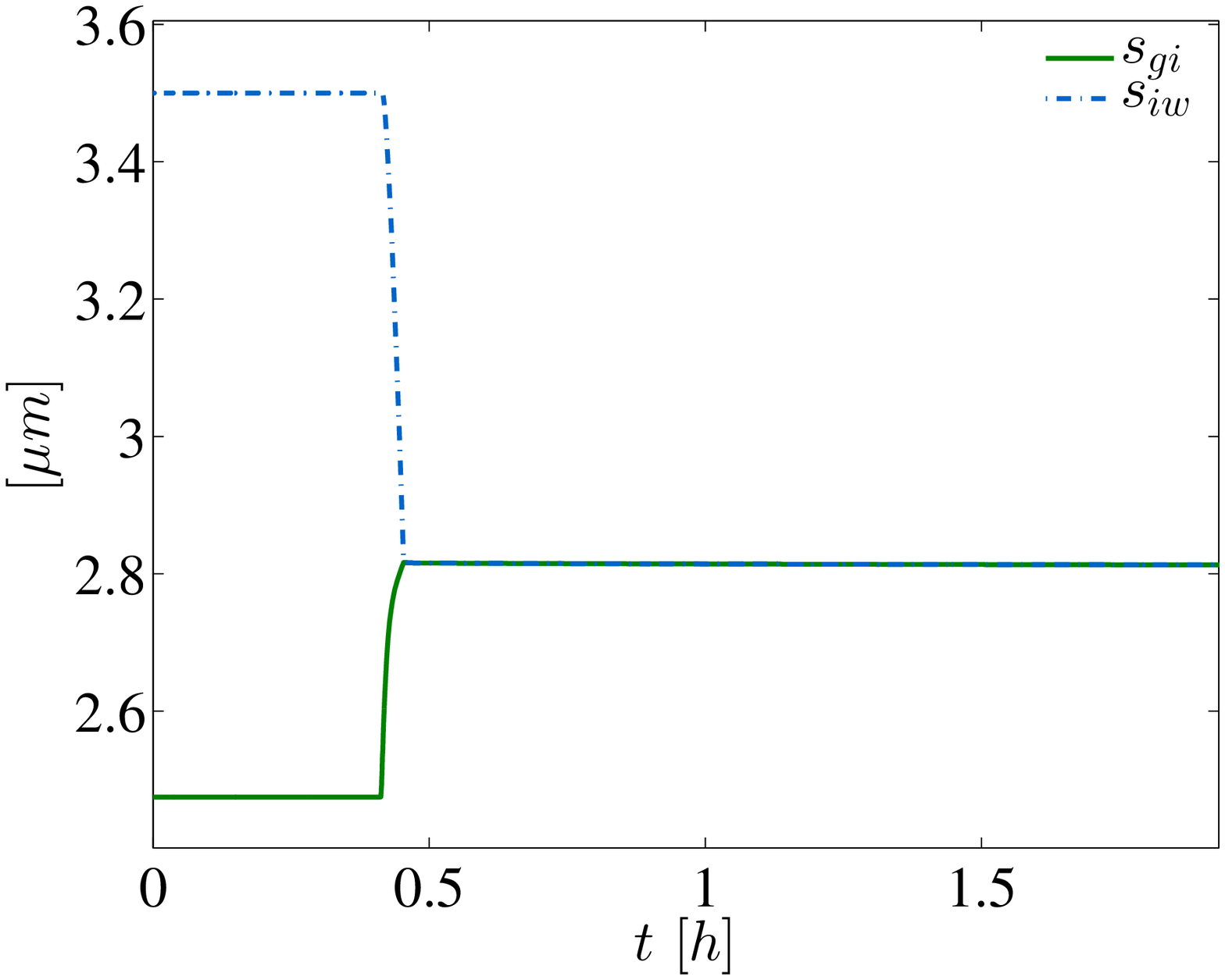} &
    \includegraphics[width=0.30\textwidth]{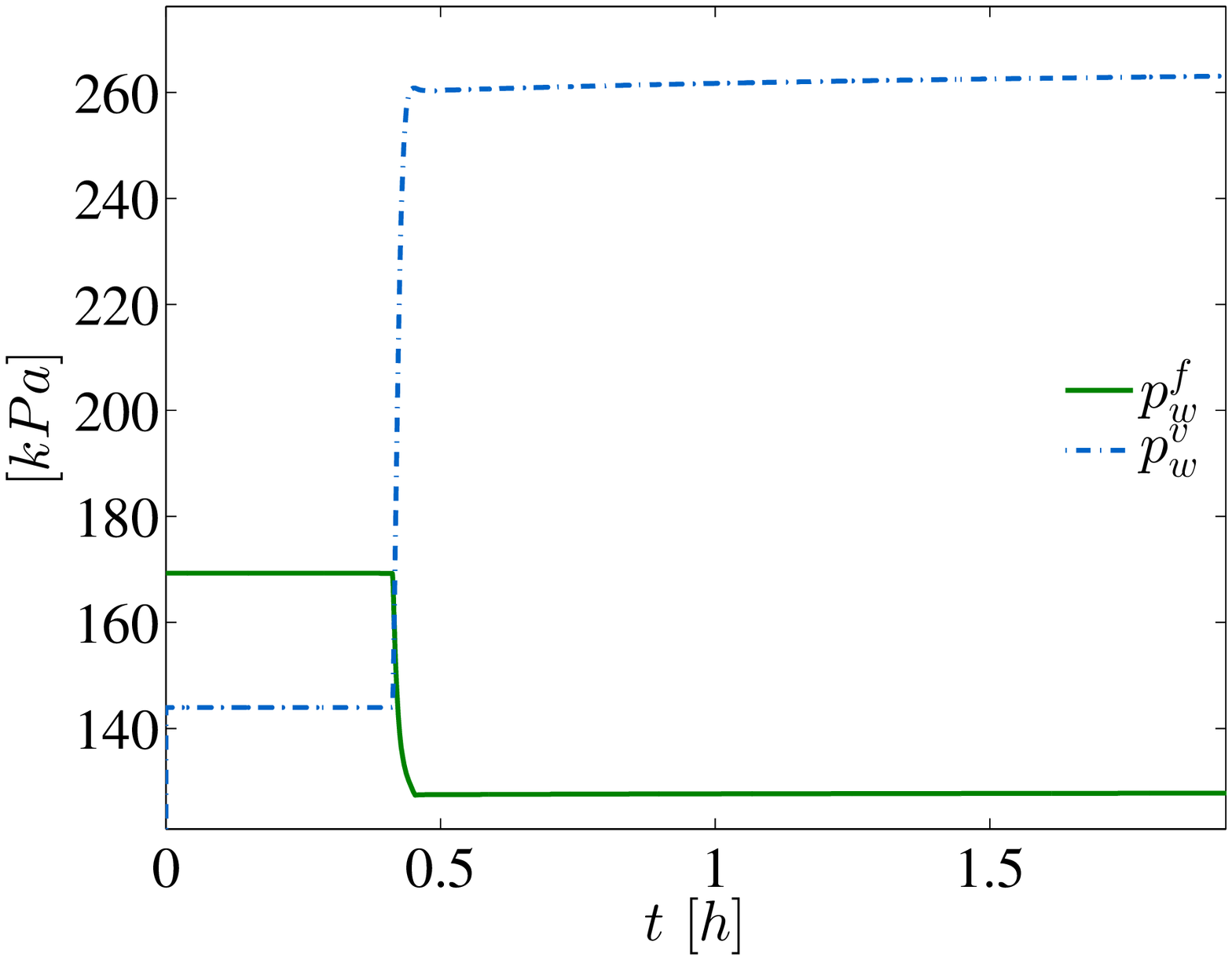} 
  \end{tabular}
  \caption{Sap exudation simulations, showing the time evolution of
    various solution components at radial position $x=0.15\;m$.}
  \label{fig:sap2}
\end{figure}

Upon more careful inspection of the gas/ice and ice/water interface
plots in Figure~\ref{fig:sap2}(b), we observe that there is a slight
time delay in the motion of $s_{iw}$ relative to $s_{gi}$.  Indeed, the
ice begins to melt at the gas/ice interface (leading to an increase in
$s_{gi}$) at a time roughly $25\;s$ in advance of when $s_{wi}$ starts
to drop, which is when a water layer appears between the ice and fiber
wall.  This phenomenon can be explained as follows.  When melt-water
first appears in a particular fiber, the gas bubble pressure is so high
that water is immediately forced out into the vessel, leaving the ice
layer in contact with the fiber wall.  The gas pressure then declines
until approximately $25\;s$ elapses, at which time the rate of water
melting exceeds that of the porous outflow and a water layer begins to
accumulate along the fiber wall.  By this time, roughly half of the
water volume contained in the fiber has been transferred into the
vessel.

One of the most important outputs of our sap exudation model is the
prediction that the vessel water pressure increases by roughly
$120\,kPa$ to a value of $p_w^v=263\;kPa$.  This is related to the
exudation pressure if were able to incorporate freezing effects and
simulate the tree over several freeze/thaw cycles.  It would therefore
be most interesting to compare this pressure increase to experimental
measurements of maple trees during the sap harvest season.

Finally, we draw a comparison between results from the sap exudation
problem for temperature $T_1$ and ice layer thickness $s_{iw}-s_{gi}$,
and the corresponding solutions we obtained for the reduced problem in
Section~\ref{sec:problem1}.  Although the shape of the temperature and
ice interface profiles are similar, there is a significant difference in
that the melting process for the reduced problem takes over 10 times
longer than for the sap exudation problem even though the macroscopic
domain and outer temperature are the same.  This discrepancy may seem at
first to be surprisingly large, but can be attributed to two causes.
First of all, the volume of ice to be melted is more than twice as large
in the reduced model since the fibers in the sap exudation problem
contain only an annular ring of ice and not a solid circle.  Secondly,
the diffusion coefficient for the sap exudation problem is roughly 10
times higher because of the much larger value of thermal diffusivity
($k/\rho c$) in the gas phase (with an upper bound of $2\times
10^{-5}\;m^2/s$ based on atmospheric conditions) compared with the
corresponding values for ice and water ($1.2\times10^{-6}$ and
$1.3\times 10^{-7}\;m^2/s$ respectively) which are the only phases
appearing in the reduced model.

\section{Conclusions}
\label{sec:discussion}

Our aim in this paper was two-fold: first to apply the techniques of
periodic homogenization to a two-phase flow problem involving a regular
array of melting ice bars that we called the ``reduced model''; and
second to apply the homogenized equations to study a much more complex
three-phase flow problem arising in the context of tree sap exudation.
The reduced model is an example of a well-known Stefan problem, and we
prove results on existence, uniqueness and a~priori estimates for the
weak form of the governing equations, which we then use to derive a
strong form of the homogenized limit problem.  The approach we employ
here has the advantage that it applies homogenization techniques in a
straightforward manner in order to obtain an uncomplicated limit model,
that in turn leads naturally to a simple numerical scheme.
Additionally, we are able to encapsulate all microscale processes
specific to the maple sap application within the domain $Y^2$, wherein
the temperature diffuses slowly.  We propose a numerical algorithm that
captures heat transport on the macroscale, while also incorporating the
phase change dynamics taking place on the microscale by solving a
corresponding reference cell problem at each point in the domain.

These homogenization equations are then applied to the study of maple
sap exudation, which is complicated by a number of extra physical
effects including the existence of a gas phase, osmotic pressure, and
porous flow between sapwood compartments.  The primary novelty in this
paper (relative to other work on homogenization of Stefan-type problems)
derives from our directly imposing the Dirichlet condition on
temperature at the phase interface, which gives rise to a decomposition
into fast/slow variables on the sub-regions $Y^1$ and $Y^2$.  The major
advantage of this approach is that the homogenized equations for the
macroscale temperature in the simpler reduced model apply equally well
to the sap exudation problem, whereas all remaining physical effects are
restricted to the reference cell problem in $Y^2$.  Moreover, this
decomposition leads immediately to a simple and efficient numerical
method that we demonstrate is capable of reproducing physically
realistic sap exudation pressures.

There are several natural possibilities for future work that arise from
this study.  In terms of gaining a complete understanding of the sap
exudation problem, it is essential to generalize the model to handle the
freezing half of the Milburn-O'Malley process, and then to simulate a
sequence of daily freeze/thaw cycles.  Only then will we be able to see
whether this process is capable of capturing the gradual build-up of
exudation pressure that is observed in experiments over a period of
several days~\cite{AHORN3}.  As part of this work, we will also
incorporate the effect of freezing point depression, implement a Robin
boundary condition on the tree circumference, and include the dynamics
of gas dissolution and nucleation.

It would also be interesting to consider the extension of our analytical
results in Section~\ref{sec:problem1} to handle the Robin type boundary
condition mentioned above.  This has been already done for an unrelated
application from T-cell signalling~\cite{graf-peter-2014}, and the
extension is straightforward but technically difficult.  The Robin
condition introduces an extra boundary term in
equation~\eqref{s_problem_weak_1} through the integration by parts
process, which in turn requires an additional one-dimensional
homogenization step along the boundary of the domain.


\section*{Appendices}
\label{appendix}

The following appendices contain proofs of the lemmas and theorems
introduced in Section~\ref{sec:problem1}.  Throughout, we use $\bigC$ or
$C_i$ to refer to a generic, real, positive constant whose value may
change from line to line.

\appendix

\section{Proofs of existence results}\label{app:existence}

We begin by stating and proving a Lemma~\ref{app:lemma} that is not
mentioned earlier in this paper but is required in the proof of
Lemma~\ref{lemma_existence}.

\begin{lemma}\label{app:lemma}
  Let $\nabla v\cdot \normvec\in \Hil^{-\half}(\Gameps)$ for a function
  $v\in \Hil^1(\Omeps)$. Then there exists a constant ${\bigC}>0$
  independent of $\eps$ such that 
  \begin{gather*}
    \eps\norm{\eps\nabla v\cdot \normvec}^2_{\Hil^{-\half}(\Gameps)}\leq
    {\bigC}\left(\norm{v}^2_{L^2(\Omeps)} + \norm{\eps\nabla
        v}^2_{L^2(\Omeps)}\right). 
  \end{gather*}
\end{lemma}

\begin{proof}
  First we observe that for functions $v\in \Hil^1(Y)$, the Neumann
  trace mapping~\cite{PDE15} guarantees that
  \begin{gather*}
    \norm{\nabla v\cdot \normvec}_{\Hil^{-\half}(\Gamma)}\leq
    \bigC\left(\norm{u}_{L^2(Y)} + \norm{\nabla v}_{L^2(Y)}\right) 
  \end{gather*}
  for a constant $\bigC>0$.
  Next, let $w\in L^2(\Gameps)$ and $\phi\in \Hil^{\half}(\Gameps)$ and
  use the fact that $\Hil^{\half}\subset L^2\subset \Hil^{-\half}$ is a
  Gelfand triple to obtain
  \begin{align*}
    \langle w,\, \phi\rangle_{\Hil^{-\half}(\Gameps)\times
      \Hil^{\half}(\Gameps)} &= (w,\, \phi)_{L^2(\Gameps)} =
    \int_{\Gameps}w\phi \,\text{d}\sigma_x
    = \sum_{k\in K_\eps}\int_{\eps\Gamma}w(x)\phi(x)\,\text{d}\sigma_x\\
    &= \sum_{k\in K_\eps}\eps^{n-1}\int_\Gamma
    w(y)\phi(y)\,\text{d}\sigma_y\\
    &= \sum_{k\in K_\eps}\eps^{n-1}\langle
    w,\; \phi\rangle_{\Hil^{-\half}(\Gamma)\times
      \Hil^{\half}(\Gamma)}, 
  \end{align*}
  where $K_\eps$ refers to the cells of size $\eps$ in $\Omeps$
  and $y=\frac{x}{\eps}$.  Since $L^2(\Gameps)$ is dense in
  $\Hil^{-\half}(\Gameps)$, the identity also holds for $v\in
  \Hil^{-\half}(\Gameps)$ so that
  \begin{gather*}
    \langle v,\, \phi\rangle_{\Hil^{-\half}(\Gameps)\times
      \Hil^{\half}(\Gameps)} = \sum_{k\in K_\eps}\eps^{n-1}\langle
    v,\, \phi\rangle_{\Hil^{-\half}(\Gamma)\times \Hil^{\half}(\Gamma)}.
  \end{gather*}
  
  We can then show that for any $\phi\in \Hil^{\half}(\Gameps)$, 
  \begin{align*}
    \norm{\phi}^2_{\Hil^{\half}(\Gameps)} &=
    \int_{\Gameps}\phi(x)\,\text{d}\sigma_x +
    \int_{\Gameps}\int_{\Gameps}\frac{|\phi(x_1) -
      \phi(x_2)|^2}{|x_1-x_2|^n}
    \,\text{d}\sigma_{x_1}\,\text{d}\sigma_{x_2}\\ 
    &= \sum_{k\in
      K_\eps}\eps^{n-1}\int_{\Gamma}\phi(y)\,\text{d}\sigma_y +
    \eps^{2n-2}\int_\Gamma\int_\Gamma\frac{|\phi(y_1) -
      \phi(y_2)|^2}{|y_1-y_2|^n}
    \,\text{d}\sigma_{y_1}\,\text{d}\sigma_{y_2}\\ 
    &= \eps^{n-1}\norm{\phi}^2_{L^2(\Gamma)} +
    \eps^{2n-2}\abs{\phi}^2_{\Hil^{\half}(\Gamma)}. 
  \end{align*}
  Finally, we come to the main step of the proof in which we estimate
  \begin{align*}
    \norm{\eps\nabla_x v\cdot \normvec}_{\Hil^{-\half}(\Gameps)} &=
    \sup_{\phi\neq 0}\frac{\langle \eps\nabla_x v(x)\cdot
      \normvec,\; \phi(x)\rangle_{\Hil^{-\half}(\Gameps)\times
        \Hil^{\half}(\Gameps)}}{\norm{\phi}_{\Hil^{\half}(\Gameps)}}\\
    &= \sup_{\phi\neq 0}\sum_{k\in
      K_\eps}\frac{\eps^{n-1}\langle\nabla_y v(y)\cdot
      \normvec,\; \phi(y)\rangle_{\Hil^{-\half}(\Gamma)\times
        \Hil^{\half}(\Gamma)}}{\eps^{\half(n-1)}\norm{\phi}_{\Hil^{\half}(\Gamma)}
      + \eps^{n-1}\abs{\phi}_{\Hil^{\half}(\Gamma)}}\\
    &= \sup_{\phi\neq 0}\sum_{k\in
      K_\eps}\frac{\eps^{n-1}\norm{\nabla_y v\cdot
        \normvec}_{\Hil^{-\half}(\Gamma)}
      \norm{\phi}_{\Hil^{\half}(\Gamma)}}{\eps^{\half(n-1)}
      \norm{\phi}_{\Hil^{\half}(\Gamma)}
      + \eps^{n-1}\abs{\phi}_{\Hil^{\half}(\Gamma)}}\\
    &\leq \sup_{\phi\neq 0}\sum_{k\in K_\eps}\frac{\bigC\norm{v}_{L^2(Y)}
      + \bigC\norm{\nabla_y v}_{L^2(Y)}}{\eps^{\half(1-n)} +
      \frac{\abs{\phi}_{\Hil^{\half}(\Gamma)}}{\norm{\phi}_{\Hil^{\half}(\Gamma)}}}\\
    &\leq \frac{\sum_{k\in K_\eps}\bigC\left(\sqrt{\int_{\eps Y} v^2 \,
          \frac{\text{d}x}{\eps^n}} + \sqrt{\int_{\eps Y}
          \eps^2|\nabla_xv|^2 \,
          \frac{\text{d}x}{\eps^n}}\right)}{\eps^{\half(1-n)} + \bigC_1}\\
    &= \frac{\eps^{-\half n}\bigC\left(\norm{v}_{L^2(\Omeps)} +
        \norm{\eps\nabla_xv}_{L^2(\Omeps)}\right)}{\eps^{\half(1-n)}
      + \bigC_1}\\
    &\leq \eps^{-\half}\bigC\left(\norm{v}_{L^2(\Omeps)} +
      \norm{\eps\nabla_xv}_{L^2(\Omeps)}\right).
  \end{align*}
  Squaring both sides of the inequality yields the desired result.
\end{proof}

\replemma{lemma_existence}{%
  The operator $\A:V\rightarrow V'$ is
 \begin{enumerate}[label=\alph*)]
  \item continuous,
  \item bounded, and
  \item coercive, in the sense that there are real constants
    $\lambda,\alpha>0$ such that 
 \begin{gather*}
    \A(t,v)(v) + \lambda \norm{v}_H^2\geq \alpha\norm{v}^2_V
  \end{gather*}
  for almost every $v\in V$ and $t\in [0,\tend]$.
 \end{enumerate}
}

\begin{proof}
  \begin{enumerate}[label=\alph*)]
  \item To prove continuity we take any $\tilde{\eps}>0$, require that
    $\norm{u-\tilde{u}}_V<\delta$ for some $\delta>0$, and then show how
    to find $\delta$ so that $\norm{\A u-\A\tilde{u}}_{V'}<\tilde{\eps}$.
    We show that for every $v\in V$ with $\norm{v}_V\leq 1$,
    \begin{align*}
      (\A u - \A\tilde{u})v &= (D(u_1)\omega'(u_1)\nabla u_1 -
      D(\tilde{u}_1)\omega'(\tilde{u}_1)\nabla\tilde{u}_1,\; \nabla
      v_1)_{L^2(\Omega_\eps^1)} \\ 
      & \qquad + \eps^2\langle D(u_2)\omega'(u_2)\nabla u_2 -
      D(\tilde{u}_2)\omega'(\tilde{u}_2)\nabla \tilde{u}_2 ,\;
      v_1\rangle_{\Hil^{-\frac12}(\Gamma_\eps)\times
        \Hil^{\frac12}(\Gamma_\eps)} \\ 
      & \qquad + \eps^2(D(u_2)\omega'(u_2)\nabla u_2 -
      D(\tilde{u}_2)\omega'(\tilde{u}_2)\nabla\tilde{u}_2,\; \nabla
      v_2)_{L^2(\Omega_\eps^2)}\\ 
      & \leq \norm{D(u_1)\omega'(u_1) -
        D(\tilde{u}_1)\omega'(\tilde{u}_1)}_{L^\infty(\Omega_\eps^1)}\norm{\nabla
        u_1}_{L^2(\Omega_\eps^1)}\norm{\nabla v_1}_{L^2(\Omega_\eps^1)} \\
      & \qquad + \norm{D(\tilde{u}_1)
        \omega'(\tilde{u}_1)}_{L^\infty(\Omega_\eps^1)} \norm{\nabla u_1 -
        \nabla \tilde{u}_1}_{L^2(\Omega_\eps^1)} \norm{\nabla
        v_1}_{L^2(\Omega_\eps^1)}\\
      & \qquad + \eps^2\norm{D(u_2)\omega'(u_2) -
        D(\tilde{u}_2)\omega'(\tilde{u}_2)}_{L^\infty(\Gamma_\eps)}
      \norm{\nabla u_2\cdot \normvec}_{\Hil^{-\frac12}(\Gamma_\eps)}
      \norm{v_1}_{\Hil^{\frac12}(\Gamma_\eps)}
      \\
      & \qquad + \eps^2\norm{D(\tilde{u}_2)
        \omega'(\tilde{u}_2)}_{L^\infty(\Gamma_\eps)} \norm{\nabla u_2\cdot
        \normvec - \nabla\tilde{u}_2\cdot
        \normvec}_{\Hil^{-\frac12}(\Gamma_\eps)}
      \norm{v_1}_{\Hil^{\frac12}(\Gamma_\eps)}\\ 
      & \qquad + \eps^2\norm{D(u_2)\omega'(u_2) - D(\tilde{u}_2)
        \omega'(\tilde{u}_2)}_{L^\infty(\Omega_\eps^2)} \norm{\nabla
        u_2}_{L^2(\Omega_\eps^2)} \norm{\nabla
        v_2}_{L^2(\Omega_\eps^2)} \\
      & \qquad + \eps^2\norm{D(\tilde{u}_2)
        \omega'(\tilde{u}_2)}_{L^\infty(\Omega_\eps^2)} \norm{\nabla u_2 -
        \nabla \tilde{u}_2}_{L^2(\Omega_\eps^2)} \norm{\nabla
        v_2}_{L^2(\Omega_\eps^2)}.
    \end{align*}
    We next apply the trace inequality $\abs{\Omtwoeps}<\infty$ and use
    the fact that $\omega'$ and $D$ are continuous and bounded to
    conclude that there exists $\delta_{D}\leq\delta$ such that
    $\norm{D(u_i) \omega'(u_i) - D(\tilde{u}_i)
      \omega'(\tilde{u}_i)}_{L^\infty(\Omega_\eps^i)} < \delta$ and
    $\norm{D(u_i)\omega'(u_i) -
      D(\tilde{u}_i)\omega'(\tilde{u}_i)}_{L^\infty(\Gamma_\eps)} <
    \delta$ when $\norm{u_i-\tilde{u}_i}_V<\delta_{D}$, for $i=1,2$.  We
    may then make use of the fact that $\Norm{u_i-\tilde{u}_i}_V<\delta$
    and $\Norm{v}_V\leq 1$ to obtain
    \begin{align*}
      (\A u - \A\tilde{u})v &\leq \delta \left(\norm{\nabla
          u_1}_{L^2(\Omega_\eps^1)} + \norm{D(\tilde{u}_1)
          \omega'(\tilde{u}_1)}_{L^\infty(\Omega_\eps^1)} \right. \\
      & \qquad + \eps^2\norm{\nabla u_2\cdot
        \normvec}_{\Hil^{-\frac12}(\Gamma_\eps)} + \eps^2
      \norm{D(\tilde{u}_2)\omega'(\tilde{u}_2)}_{L^\infty(\Gamma_\eps)} \\
      & \qquad \left. + \eps^2\norm{\nabla u_2}_{L^2(\Omega_\eps^2)} +
        \eps^2 \norm{D(\tilde{u}_2)
          \omega'(\tilde{u}_2)}_{L^\infty(\Omega_\eps^2)} \right) \\
      & = \delta \alpha \leq \tilde{\eps}.
    \end{align*}
    The last inequality is true so long as $\delta <\frac{\tilde{\eps}}\alpha$ .
    
  \item To show that $\A$ is bounded, we prove that there exists a constant
    $\bigC>0$ such that 
    \begin{gather}
      \label{app-Auv}
      \A u(v)\leq \bigC \Norm{u}_V\Norm{v}_V \qquad\text{for any}\
      u,\;v\in V, 
    \end{gather}
    which we can derive using the follow sequence of estimates:
    \begin{align*}
      \A u(v) &\leq \max\{D\omega'\} \norm{\nabla
        u_1}_{L^2(\Omega_\eps^1)} \norm{\nabla v_1}_{L^2(\Omega_\eps^1)} +
      \eps^2 \max\{D\omega'\} \norm{\nabla u_2\cdot
        \normvec}_{\Hil^{-\frac12}(\Gamma_\eps)}
      \norm{v_1}_{\Hil^{\frac12}(\Gamma_\eps)}\\ 
      & \qquad + \eps^2\max\{D\omega'\} \norm{\nabla
        u_2}_{L^2(\Omega_\eps^2)} \norm{\nabla v_2}_{L^2(\Omega_\eps^2)}\\
      & \leq \max\{D\omega'\} \norm{\nabla u_1}_{L^2(\Omega_\eps^1)}
      \norm{\nabla v_1}_{L^2(\Omega_\eps^1)} + \eps^2\max\{D\omega'\}
      \norm{\nabla u_2}_{L^2(\Omega_\eps^2)} \norm{\nabla
        v_2}_{L^2(\Omega_\eps^2)}\\ 
      & \qquad + \max\{D\omega'\}c_0 \left(\norm{u_2}_{L^2(\Omega_\eps^2)}
        + \norm{\eps\nabla u_2}_{L^2(\Omega_\eps^2)}
      \right)\left(\norm{v_1}_{L^2(\Omega_\eps^1)} + \norm{\eps\nabla
          v_1}_{L^2(\Omega_\eps^1)}\right)\\ 
      & \leq \left( \max\{D\omega'\} + c_0\max\{D\omega'\} +
        \eps^2\max\{D\omega'\}\right) \norm{u}_V\norm{v}_V.
    \end{align*}
    Hence, the operator $\A$ is bounded as in \eqref{app-Auv} for every
    $0<\eps<1$ with the constant
    \begin{gather*}
      \bigC = \max\{D\omega'\}  + c_0\max\{D\omega'\}  + \max\{D\omega'\}.
    \end{gather*}
    
  \item For coercivity we proceed by deriving the following estimates for
    \begin{align*}
      \A v(v) + \lambda \norm{v}_H^2 &= (D(v_1)\omega'(v_1)\nabla v_1 ,\;
      \nabla v_1)_{\Omega_\eps^1} + \langle \eps^2D(v_2) \omega'(v_2)
      \nabla v_2\cdot \normvec ,\;
      v_1\rangle_{\Hil^{-\frac12}(\Gamma_\eps)
        \times\Hil^{\frac12}(\Gamma_\eps)} \\
      & \qquad + (\eps^2D(v_2)\omega'(v_2)\nabla v_2,\; \nabla
      v_2)_{\Omega_\eps^2} + \lambda\norm{v_1}^2_{\Omega_\eps^1} +
      \lambda\norm{v_2}^2_{\Omega_\eps^2} \\ 
      & \geq \min\{D\omega'\}\norm{\nabla v_1}^2_{\Omega_\eps^1} 
      - \max\{D\omega'\}c_0\left(\norm{v_2}^2_{\Omega_\eps^2} +
        \eps^2\norm{\nabla v_2}^2_{\Omega_\eps^2} +
        \norm{v_1}^2_{\Omega_\eps^1} + \eps^2\norm{\nabla
          v_1}^2_{\Omega_\eps^1} \right)\\ 
      & \qquad  + \min\{D\omega'\}\eps^2\norm{\nabla
        v_2}^2_{\Omega_\eps^2} + \lambda\norm{v_1}^2_{\Omega_\eps^1} +
      \lambda\norm{v_2}^2_{\Omega_\eps^2}\\ 
      & \geq \norm{v_1}^2_{\Omega_\eps^1}\left( \lambda  -
        \max\{D\omega'\}c_0 \right) + \norm{\nabla
        v_1}^2_{\Omega_\eps^1}\left(\min\{D\omega'\}   -
        \max\{D\omega'\}c_0\eps^2\right)\\ 
      & \qquad + \norm{v_2}^2_{\Omega_\eps^2}\left(\lambda -
        \max\{D\omega'\}c_0 \right) + \norm{\nabla v_2}^2_{\Omega_\eps^2}
      \eps^2\left( \min\{D\omega'\}  - \max\{D\omega'\}c_0 \right) \\
			& \geq \alpha\norm{v}^2_V,
    \end{align*}
    as required.  We have merged all constants from the second-to-last
    expression into a single constant $\alpha>0$, which is possible
    provided that $( \min\{D\omega'\} - \max\{D\omega'\}c_0 ) > 0$, and
    after choosing both $\eps$ small enough and $\lambda$ big enough. We
    also applied Lemma~\ref{app:lemma} in the derivation of the second
    inequality above.
  \end{enumerate}
\end{proof}

\section{Proofs of results related to a~priori estimates}\label{app:estimates}

\replemma{lemma_estimates}{%
Here, $C$ or $C_i$ represents a positive constant that is independent of $\eps$. 
  \begin{enumerate}[label=\alph*)]
  \item  There exists a constant $\bigC_1$ such that  
    \begin{gather*}
      \Norm{\Toneeps}^2_{\Omoneeps} + \Norm{\nabla
        \Toneeps}^2_{\Omoneeps} + \Norm{\Ttwoeps}^2_{\Omtwoeps}
      + \eps^2\Norm{\nabla \Ttwoeps}^2_{\Omtwoeps} \leq \bigC_1 . 
    \end{gather*}
    
  \item There exists a constant $\bigC_2$ such that 
    \begin{multline*}
      \Norm{\nabla \HEoneeps}^2_{\Omoneeps} +
      \Norm{\nabla\cdot[D(\HEoneeps)\nabla \Toneeps]}^2_{\Omoneeps,t} +
      \eps^2\Norm{\nabla \HEtwoeps}^2_{\Omtwoeps} + 
      \eps^4\Norm{\nabla\cdot[D(\HEtwoeps)\nabla
        \Ttwoeps]}^2_{\Omtwoeps,t} \leq \bigC_2.
    \end{multline*}

  \item There exists a constant $\bigC_3$ such that 
    \begin{gather*}
      \eps^3\Norm{D(\HEtwoeps)\nabla \Ttwoeps}^2_{\Gameps,t} \leq
      \bigC_3. 
    \end{gather*}

    
  \item The functions $\HEoneeps$, $\HEtwoeps$ are nonnegative for
    almost every $x\in\Omoneeps$, $\Omtwoeps$ as 
    long as $\HEoneeps(0)$, $\HEtwoeps(0)$ are nonnegative
    (respectively). 
    
  \item There exists a constant $\bigC_4$ such that  
    \begin{gather*}
      \Norm{\HEoneeps}_{L^\infty(\Omoneeps)} +
      \Norm{\HEtwoeps}_{L^\infty(\Omtwoeps)} \leq \bigC_4. 
    \end{gather*}
    
\item  There exists a constant $\bigC_5$ such that 
  \begin{gather*}
    \Norm{\partial_t \HEoneeps}_{L^2([0,\tend],L^2(\Omoneeps))} +
    \Norm{\partial_t \HEtwoeps}_{L^2([0,\tend],L^2(\Omtwoeps))}
    \leq \bigC_5. 
  \end{gather*}
\end{enumerate}
}

\begin{proof}
  \begin{enumerate}[label=\alph*)]
  \item We test equation~\eqref{s_problem_weak} with
    $(\Toneeps,\Ttwoeps)$ to obtain
    \begin{align*}
      (\partial_t\HEoneeps,\; \Toneeps)_{\Omoneeps} +
      (D(\HEoneeps)\nabla \Toneeps,\; \nabla \Toneeps)_{\Omoneeps} +
      \eps^2\langle D(\HEtwoeps)\nabla \Ttwoeps\cdot \normvec,\;
      \Toneeps\rangle_{\Gamma_\eps} &= 0,\\ 
      (\partial_t\HEtwoeps,\; \Ttwoeps)_{\Omtwoeps} +
      \eps^2(D(\HEtwoeps)\nabla \Ttwoeps,\; \nabla \Ttwoeps)_{\Omtwoeps}
      - \eps^2\langle D(\HEtwoeps)\nabla \Ttwoeps\cdot \normvec,\;
      \Toneeps\rangle_{\Gamma_\eps} &= 0, 
    \end{align*}
    where we have the boundary term in the second equation because
    $\Ttwoeps=\Toneeps$ is not zero at $\Gamma_\eps$.  The two equations
    may then be added to eliminate the boundary terms, after which we
    rewrite the enthalpy variables in terms of temperature using chain
    rule: 
    \begin{multline*}
      ((\omega^{-1})'\partial_t\Toneeps,\; \Toneeps)_{\Omoneeps} +
      \min\{D\}(\nabla \Toneeps,\; \nabla \Toneeps)_{\Omoneeps} +
      ((\omega^{-1})'\partial_t\Ttwoeps,\; \Ttwoeps)_{\Omtwoeps}
      \\
      + \eps^2\min\{D\}(\nabla \Ttwoeps,\; \nabla \Ttwoeps)_{\Omtwoeps} 
      = 0 .
    \end{multline*}
    Next, using that $(\omega^{-1})'$ and $D$ are both bounded and
    positive, we can integrate in time to obtain
    \begin{multline*}
      \half\min\{(\omega^{-1})'\}\norm{\Toneeps}^2_{\Omoneeps}
      + \min\{D\}\norm{\nabla \Toneeps}^2_{\Omoneeps,t} +
      \half\min\{(\omega^{-1})'\}\norm{\Ttwoeps}^2_{\Omtwoeps}
      + \min\{D\}\norm{\nabla \Ttwoeps}^2_{\Omtwoeps,t}\\ 
      \leq
      \half\max\{(\omega^{-1})'\}\norm{\Toneeps(0)}^2_{\Omoneeps}
      +
      \half\max\{(\omega^{-1})'\}\norm{\Ttwoeps(0)}^2_{\Omtwoeps} .
    \end{multline*}
    Since the initial conditions are bounded, we obtain by 
    merging the constants that
    \begin{align*}
      \norm{\Toneeps}^2_{\Omoneeps} + \norm{\nabla
        \Toneeps}^2_{\Omoneeps} + \norm{\Ttwoeps}^2_{\Omtwoeps} +
      \eps^2\norm{\nabla \Ttwoeps}^2_{\Omtwoeps} \leq \bigC_1 ,
    \end{align*}
    for a constant $\bigC_1>0$ that is independent of $\eps$. With
    $\omega$ and $\omega^{-1}$ being continuous and bounded, and with
    $\HEoneeps = \omega^{-1}(\Toneeps)$ and
    $\HEtwoeps=\omega^{-1}(\Ttwoeps)$, it follows that
    $\norm{\HEoneeps}^2_{\Omoneeps} + \norm{\HEtwoeps}^2_{\Omtwoeps}
    \leq \bigC_1$.
    
  \item We test equation~\eqref{s_problem_weak} with
    $(\nabla\cdot[D(\HEoneeps)\nabla \Toneeps],
    \eps^2\nabla\cdot[D(\HEtwoeps)\nabla \Ttwoeps])$, which can be made
    rigorous as in Lemma~7.23 of \cite{PDE14}:
    \begin{align*}
      (\partial_t\HEoneeps,\; \nabla\cdot[D(\HEoneeps)\nabla
      \Toneeps])_{\Omoneeps}  - \norm{\nabla\cdot[D(\HEoneeps)\nabla
        \Toneeps]}^2_{\Omega_\eps^ 1} &= 0,\\ 
      \eps^2(\partial_t\HEtwoeps,\; \nabla\cdot[D(\HEtwoeps)\nabla
      \Ttwoeps])_{\Omtwoeps} -
      \eps^4\norm{\nabla\cdot[D(\HEtwoeps)\nabla
        \Ttwoeps]}^2_{\Omtwoeps} &= 0. 
    \end{align*}
    Integrating by parts in the first term of each equation yields
    \begin{align*}
      -(\partial_t\nabla \HEoneeps,\; D(\HEoneeps)\nabla
      \Toneeps)_{\Omoneeps} - \eps^2\langle \partial_t\HEoneeps,\;
      D(\HEtwoeps)\nabla \Ttwoeps\cdot \normvec\rangle_{\Gamma_\eps} -
      \norm{\nabla\cdot[D(\HEoneeps)\nabla \Toneeps]}^2_{\Omega_\eps^ 1}
      &= 0,\\ 
      -\eps^2(\partial_t\nabla \HEtwoeps,\; D(\HEtwoeps)\nabla
      \Ttwoeps)_{\Omtwoeps} + \eps^2\langle \partial_t\HEoneeps,\;
      D(\HEtwoeps)\nabla \Ttwoeps\cdot \normvec\rangle_{\Gamma_\eps} -
      \eps^4\norm{\nabla\cdot[D(\HEtwoeps)\nabla
        \Ttwoeps]}^2_{\Omtwoeps} &= 0, 
    \end{align*}
    where we have also applied the boundary conditions on
    $\Gamma_\eps$. Adding the last two equations together and swapping
    sign yields 
    \begin{multline*}
      (\partial_t\nabla \HEoneeps,\;
      D(\HEoneeps)\omega'(\HEoneeps)\nabla \HEoneeps)_{\Omoneeps} +
      \norm{\nabla\cdot[D(\HEoneeps)\nabla \Toneeps]}^2_{\Omoneeps}\\ 
      + \eps^2(\partial_t\nabla \HEtwoeps,\;
      D(\HEtwoeps)\omega'(\HEtwoeps) \nabla \HEtwoeps)_{\Omtwoeps} +
      \eps^4\norm{\nabla\cdot[D(\HEtwoeps)\nabla
        \Ttwoeps]}^2_{\Omtwoeps} = 0 .
    \end{multline*}
    We then integrate with respect to time and use the fact that
    $\omega'$ and $D$ are bounded and positive to obtain  
    \begin{align*}
      \half\min\{D\omega'\}\norm{\nabla
        \HEoneeps}^2_{\Omoneeps} &+
      \norm{\nabla\cdot[D(\HEoneeps)\nabla \Toneeps]}^2_{\Omoneeps,t} +
      \frac{\eps^2}{2}\textstyle\min\{D\omega'\}\norm{\nabla
        \HEtwoeps}^2_{\Omtwoeps} +
      \eps^4\norm{\nabla\cdot[D(\HEtwoeps)\nabla
        \Ttwoeps]}^2_{\Omtwoeps,t} \\ 
      &\leq \half\max\{D\omega'\}\norm{\nabla
        \HEoneeps(0)}^2_{\Omoneeps} +
      \frac{\eps^2}{2}\max\{D\omega'\}\norm{\nabla
        \HE_{2,\eps,t}(0)}^2_{\Omtwoeps}\\ 
      &\leq C,
    \end{align*}
    where we combined the smooth and bounded initial conditions into a
    single constant $C$.  Merging all constants into a single
    $\bigC_2>0$ yields 
    \begin{gather*}
      \norm{\nabla \HEoneeps}^2_{\Omoneeps} +
      \norm{\nabla\cdot[D(\HEoneeps)\nabla \Toneeps]}^2_{\Omoneeps,t} +
      \eps^2\norm{\nabla \HEtwoeps}^2_{\Omtwoeps} +
      \eps^4\norm{\nabla\cdot[D(\HEtwoeps)\nabla
        \Ttwoeps]}^2_{\Omtwoeps,t} \leq \bigC_2.
    \end{gather*}
    
  \item The proof is a straightforward application of parts a) and b)
    along with the trace inequality, which yield 
    \begin{align*}
      \eps^3\norm{D(\HEtwoeps)\nabla \Ttwoeps}^2_{\Gamma_\eps,t} \leq
      \eps^2\norm{D(\HEtwoeps)\nabla \Ttwoeps}^2_{\Omtwoeps,t} +
      \eps^4\norm{\nabla\cdot[D(\HEtwoeps)\nabla
        \Ttwoeps]}^2_{\Omtwoeps,t} 
      \leq \bigC_3.
    \end{align*}
    
    
  \item We test equation~\eqref{s_problem_weak} with
    $(\HE_{1,\eps-}, \ \HE_{2,\eps-})$, where $\HE_{1,\eps-} =
    -\HEoneeps$ if $\HEoneeps\leq0$ pointwise and $\HE_{1,\eps-}=0$
    otherwise; an analogous definition is used for $\HE_{2,\eps-}$.  We
    then find that
    \begin{align*}
      (\partial_t \HE_{1,\eps-},\; \HE_{1,\eps-})_{\Omoneeps} +
      (D(\HE_{1,\eps-})\nabla \Toneeps,\; \nabla
      \HE_{1,\eps-})_{\Omoneeps} + \eps^2\langle D(\HEtwoeps\nabla
      \Ttwoeps,\; \HE_{1,\eps-}\rangle_{\Gamma_\eps} &= 0,\\ 
      (\partial_t \HE_{2,\eps-},\; \HE_{2,\eps-})_{\Omtwoeps} +
      \eps^2(D(\HE_{2,\eps-})\nabla \Ttwoeps,\; \nabla
      \HE_{2,\eps-})_{\Omtwoeps} - \eps^2\langle D(\HEtwoeps\nabla
      \Ttwoeps,\; \HE_{1,\eps-}\rangle_{\Gamma_\eps} &= 0, 
    \end{align*} 
    where we have used that $\HE_{2,\eps-} = \HE_{1,\eps-}$ on
    $\Gamma_\eps$.  Add the two equations above, substitute
    $\Ttwoeps=\omega(\HEtwoeps)$ and $\Toneeps=\omega(\HEoneeps)$, and
    then integrate over the time interval $[0,t]$ to obtain
    \begin{align*}
      \textstyle \frac12 \norm{\HE_{1,\eps-}}^2_{\Omoneeps} +
      \min\{D\omega'\}\norm{\nabla \HE_{1,\eps-}}^2_{\Omoneeps,t} +
      \half\norm{\HE_{2,\eps-}}^2_{\Omtwoeps} +
      \eps^2\min\{D\omega'\}\norm{\nabla \HE_{2,\eps-}}^2_{\Omtwoeps,t}
      \leq 0, 
    \end{align*}
    where we assumed that the initial conditions are nonnegative.  This
    yields that $\HEoneeps, \HEtwoeps\geq 0$ almost everywhere, from
    which we can conclude that $\HEoneeps$, $\HEtwoeps$ are bounded from
    below in $L^2([0,\tend],L^\infty(\Omtwoeps))$, independently of
    $\eps$.
        
  \item Let $M(t) = \max\{\norm{\HEoneeps(0)}_{L^\infty(\Omoneeps)},
    \norm{\HEtwoeps(0)}_{L^\infty(\Omtwoeps)}\} e^{kt}$ for some
    constant $k\in\R$. We test equation~\eqref{s_problem_weak} with
    $((\HEoneeps - M)_+,(\HEtwoeps - M)_+)$, where $(\HEoneeps-M)_+ =
    \HEoneeps-M$ if $\HEoneeps\leq M$ pointwise and $(\HEoneeps-M)_+ =
    0$ otherwise; an analogous definition is used for
    $(\HEtwoeps-M)_+$. We then obtain that
    \begin{align*}
      (\partial_t\HEoneeps,\; (\HEoneeps - M)_+)_{\Omoneeps} +
      (D(\HEoneeps)\nabla \Toneeps,\; \nabla(\HEoneeps -
      M)_+)_{\Omoneeps} + \eps^2\langle D(\HEtwoeps)\nabla \Ttwoeps,\;
      (\HEoneeps-M)_+\rangle_{\Gamma_\eps} &= 0,\\ 
      (\partial_t\HEtwoeps,\; (\HEtwoeps - M)_+)_{\Omtwoeps} +
      \eps^2(D(\HEtwoeps)\nabla \Ttwoeps,\; \nabla(\HEtwoeps -
      M)_+)_{\Omtwoeps} - \eps^2\langle D(\HEtwoeps)\nabla \Ttwoeps,\;
      (\HEoneeps-M)_+\rangle_{\Gamma_\eps} &= 0, 
    \end{align*}
    where $(\HEoneeps - M)_+  = (\HEtwoeps - M)_+$ on $\Gamma_\eps$. 
    Next, we subtract $\partial_tM=kM$ from the first inner product term
    in both equations and the corresponding term from the right-hand
    side. Then both equations are added together so that the
    boundary-term cancels, leaving
    \begin{multline*}
      (\partial_t(\HEoneeps - M)_+,\; (\HEoneeps - M)_+)_{\Omoneeps} +
      (D(\HEoneeps)\omega'(\HEoneeps)\nabla \HEoneeps,\;
      \nabla(\HEoneeps-M)_+)_{\Omoneeps}\\ 
      + (\partial_t(\HEtwoeps - M)_+,\; (\HEtwoeps - M)_+)_{\Omtwoeps} +
      \eps^2(D(\HEtwoeps\omega'(\HE_{2\eps})\nabla \HEtwoeps,\; \nabla
      (\HEtwoeps - M)_+)_{\Omtwoeps}\\ 
      = (-kM,\; (\HEoneeps-M)_+)_{\Omoneeps} +  (-kM,\;
      (\HEtwoeps-M)_+)_{\Omtwoeps}. 
    \end{multline*}
    Because $\nabla M = 0$, we may subtract $\nabla M$ from the second
    term, while at the same time estimating $D\omega'$ to obtain
    \begin{multline*}
      (\partial_t(\HEoneeps - M)_+,\; (\HEoneeps - M)_+)_{\Omoneeps} +
      \min\{D\omega'\}(\nabla (\HEoneeps-M),\;
      \nabla(\HEoneeps-M)_+)_{\Omoneeps}\\ 
      + (\partial_t(\HEtwoeps - M)_+,\; (\HEtwoeps - M)_+)_{\Omtwoeps} +
      \eps^2\min\{D\omega'\}(\nabla (\HEtwoeps-M),\; \nabla (\HEtwoeps -
      M)_+)_{\Omtwoeps}\\ 
      = (-kM,\; (\HEoneeps-M)_+)_{\Omoneeps} +  (-kM,\;
      (\HEtwoeps-M)_+)_{\Omtwoeps}. 
    \end{multline*}
    Finally, integrating with respect to time yields
    \begin{multline*}
      \half\norm{(\HEoneeps-M)_+}^2_{\Omoneeps}
      +\min\{D\omega'\}\norm{\nabla(\HEoneeps-M)_+}^2_{\Omoneeps,t} \\ 
      + \half\norm{(\HEtwoeps-M)_+}^2_{\Omtwoeps} +
      \eps^2\min\{D\omega'\}\norm{\nabla(\HEtwoeps -
        M)_+}^2_{\Omtwoeps,t}\\ 
      \leq -(kM,\; (\HEoneeps - M)_+)_{\Omoneeps,t} - (kM,\; (\HEtwoeps
      - M)_+)_{\Omtwoeps,t} \leq 0,
    \end{multline*}
    where we used that the initial conditions are smaller that
    $M$. Hence, we find that $\norm{(\HEoneeps - M)_+}^2_{\Omoneeps} +
    \norm{(\HEtwoeps-M)_+}^2_{\Omtwoeps}\leq 0$ which means
    $\HEoneeps\leq M$ and $\HEtwoeps\leq M$ almost everywhere. 
        
    
  \item  We consider equations \eqref{s_problem_weak} before integration by parts
    \begin{align*}
      \partial_t\HEoneeps &= \nabla\cdot [D(\HEoneeps)\nabla
      \Toneeps]\qquad \text{in}\ \Omoneeps,\\ 
      \partial_t\HEtwoeps &= \nabla\cdot [D(\HEtwoeps)\nabla
      \Ttwoeps]\qquad \text{in}\ \Omtwoeps, 
    \end{align*}
    and apply the $L^2$-norm over $[0,\tend]\times\Omoneeps$ and
    $[0,\tend]\times\Omtwoeps$, respectively.  Then we can use item b)
    of this Lemma to obtain
    \begin{align*}
      \norm{\partial_t\HEoneeps}_{\Omoneeps,t} = \norm{\nabla\cdot
        [D(\HEoneeps)\nabla \Toneeps]}_{\Omoneeps,t} \leq \bigC_5,\\ 
      \norm{\partial_t\HEtwoeps}_{\Omtwoeps,t} = \norm{\nabla\cdot
        [D(\HEtwoeps)\nabla \Ttwoeps]}_{\Omtwoeps,t} \leq \bigC_5, 
    \end{align*}
    where constant $\bigC_5$ is independent of $\eps$.
  \end{enumerate}
\end{proof}

\section{Proof of uniqueness}\label{app:uniqueness}

\reptheorem{uniqueness}{%
  Equations~\eqref{s_problem_limit_1} have at most one solution given by 
  \begin{alignat*}{3}
    T_1 & \in \V^1(\Omega) + \Tout &&= 
		L^2([0,\tend], \Hil^1_0(\Omega)) \cap 
		\Hil^1([0,\tend], L^2(\Omega)) + \Tout, \\
    T_2 &\in \V^2(\Omega\times \Ytwo) +T_1 &&= 
		L^2([0,\tend],L^2(\Omega,\Hil^1_\#(\Ytwo))) \cap 
		\Hil^1([0,\tend],L^2(\Omega\times \Ytwo)) + T_1,
  \end{alignat*}
  where $T_1 = \omega(\HE_1)$ and $T_2 = \omega(\HE_2)$.
}

\begin{proof}
  First, we note that the cell problem \eqref{cell_problem1} has a
  unique solution, which is proven in \cite{TWOSCALE26}. Hence, we
  will only prove uniqueness of the macroscopic problem by assuming that
  there are two solutions $(\HE_{1,a},T_{1,a},\HE_{2,a},T_{2,a})$ and
  $(\HE_{1,b},T_{1,b},\HE_{2,b},T_{2,b})$, and then showing that they
  are equal. We start by substituting our two solutions into the second
  equation of \eqref{s_problem_limit_1}, subtracting, and then
  testing with the functions $\HE_{1,a} - \HE_{1,b}$ and $\HE_{2,a} -
  \HE_{2,b}$:
  \begin{align*}
    \abs{Y^1}(\partial_t\HE_{1,a}-\partial_t\HE_{1,b}&,\; 
    \HE_{1,a}-\HE_{1,b})_\Omega + \Pi (D(\HE_{1,a})\nabla T_{1,a} -
    D(\HE_{1,b})\nabla T_{1,b},\; \nabla \HE_{1,a}-\nabla
    \HE_{1,b})_\Omega \\ 
    & \qquad\qquad + (D(\HE_{2,a})\nabla T_{2,a}\cdot \normvec -
    D(\HE_{2,b})\nabla T_{2,b}\cdot \normvec,\; \HE_{1,a}-
    \HE_{1,b})_{\Omega\times\Gamma} = 0,\\ 
    (\partial_t\HE_{2,a} - \partial_t\HE_{2,b}&,\; \HE_{2,a} -
    \HE_{2,b})_{\Omega\times Y^2} + (D_2(\HE_{2,a})\nabla_y T_{2,a} -
    D_2(\HE_{2,b})\nabla_y T_{2,b},\; \nabla_y\HE_{2,a} -
    \nabla_y\HE_{2,b})_{\Omega\times Y^2}  \\ 
    &\qquad\qquad - (D(\HE_{2,a})\nabla T_{2,a}\cdot \normvec -
    D(\HE_{2,b})\nabla T_{2,b}\cdot \normvec,\;
    \HE_{1,a}-\HE_{1,b})_{\Omega\times\Gamma} = 0. 
  \end{align*}
  Adding these two equations, using $T=\omega(\HE)$, and
  adding/subtracting an extra term we obtain 
  \begin{multline*}
    \abs{Y^1}(\partial_t\HE_{1,a}-\partial_t\HE_{1,b},\;
    \HE_{1,a}-\HE_{1,b})_\Omega
    + \Pi (D(\HE_{1,a})\omega'(\HE_{1,a})\nabla \HE_{1,a} -
    D(\HE_{1,a})\omega'(\HE_{1,a})\nabla \HE_{1,b} \\ 
    + D(\HE_{1,a})\omega'(\HE_{1,a})\nabla \HE_{1,b} -
    D(\HE_{1,b})\omega'(\HE_{1,b})\nabla \HE_{1,b},\; \nabla
    \HE_{1,a}-\nabla \HE_{1,b})_\Omega \\ 
    + (\partial_t\HE_{2,a} - \partial_t\HE_{2,b},\; \HE_{2,a} -
    \HE_{2,b})_{\Omega\times Y^2}
    + (D(\HE_{2,a})\omega'(\HE_{2,a})\nabla_y \HE_{2,a}  -
    D(\HE_{2,a})\omega'(\HE_{2,a})\nabla \HE_{2,b} \\ 
    + D(\HE_{2,a})\omega'(\HE_{2,a})\nabla \HE_{2,b} -
    D(\HE_{2,b})\omega'(\HE_{2,b})\nabla_y \HE_{2,b},\;
    \nabla_y\HE_{2,a} - \nabla_y\HE_{2,b})_{\Omega\times Y^2} = 0 .
  \end{multline*}
  Then we can develop the following estimate
  \begin{align*}
    \abs{Y^1}(\partial_t\HE_{1,a} -& \partial_t\HE_{1,b},\; 
    \HE_{1,a}-\HE_{1,b})_\Omega + \Pi \min\{D\omega'\}\norm{\nabla
      \HE_{1,a} - \nabla \HE_{1,b}}^2_\Omega \\ 
    &\qquad\qquad +  (\partial_t\HE_{2,a} - \partial_t\HE_{2,b},\;
    \HE_{2,a} - \HE_{2,b})_{\Omega\times Y^2}  
    + \min\{D\omega'\}\norm{\nabla_y \HE_{2,a}  -\nabla
      \HE_{2,b}}^2_{\Omega\times Y^2} \\ 
    & = -  ((D(\HE_{1,a})\omega'(\HE_{1,a}) -
    D(\HE_{1,b})\omega'(\HE_{1,b}))\nabla \HE_{1,b},\; \nabla
    \HE_{1,a}-\nabla \HE_{1,b})_\Omega  \\ 
    &\qquad\qquad - ((D(\HE_{2,a})\omega'(\HE_{2,a}) -
    D(\HE_{2,b})\omega'(\HE_{2,b}))\nabla_y \HE_{2,b},\;
    \nabla_y\HE_{2,a} - \nabla_y\HE_{2,b})_{\Omega\times Y^2}\\ 
    &\leq \bigC_D\norm{(\HE_{1,a}-\HE_{1,b})\nabla
      \HE_{1,b}}_\Omega\norm{\nabla \HE_{1,a}-\nabla \HE_{1,b}}_\Omega\\ 
    &\qquad\qquad + \bigC_D\norm{(\HE_{2,a}-\HE_{2,b})\nabla_y
      \HE_{2,b}}_{\Omega\times
      Y^2}\norm{\nabla_y\HE_{2,a}-\nabla_y\HE_{2,b}}_{\Omega \times
      Y^2},  
  \end{align*}
  where we used that $D\omega'$ is Lipschitz continuous with constant
  $\bigC_D$. Now we continue applying the quadratic formula and integrating
  with respect to time 
  \begin{align*}
    \half\abs{Y^1}\norm{\HE_{1,a}&-\HE_{1,b}}^2_\Omega + \Pi
    \min\{D\omega'\}\norm{\nabla \HE_{1,a} - \nabla
      \HE_{1,b}}^2_{\Omega,t} \\
    &\qquad\qquad + \half\norm{\HE_{2,a} - \HE_{2,b}}^2_{\Omega\times
      Y^2} + \min\{D\omega'\}\norm{\nabla_y \HE_{2,a} -\nabla
      \HE_{2,b}}^2_{\Omega\times Y^2,t} \\
    &\leq \frac{\bigC_D \lambda}{2}\norm{(\HE_{1,a}-\HE_{1,b})\nabla
      \HE_{1,b}}^2_{\Omega,t} + \frac{\bigC_D}{2\lambda}\norm{\nabla
      \HE_{1,a}-\nabla \HE_{1,b}}^2_{\Omega,t}\\
    &\qquad\qquad +
    \frac{\bigC_D\lambda}{2}\norm{(\HE_{2,a}-\HE_{2,b})\nabla_y
      \HE_{2,b}}^2_{\Omega\times Y^2,t} + \textstyle\frac1{2\lambda}
    \bigC_D \norm{\nabla_y\HE_{2,a}-\nabla_y\HE_{2,b}}^2_{\Omega \times
      Y^2,t},
  \end{align*}
  for any $\lambda>0$, where we have taken advantage of the fact that
  the initial conditions cancel. Rearranging terms yields 
  \begin{align*}
    \half\abs{Y^1}\norm{\HE_{1,a}&-\HE_{1,b}}^2_\Omega +
    \left(\Pi \min\{D\omega'\} -
      \frac{\bigC_D}{2\lambda}\right)\norm{\nabla \HE_{1,a} - \nabla
      \HE_{1,b}}^2_{\Omega,t} \\ 
    &\qquad\qquad + \half\norm{\HE_{2,a} - \HE_{2,b}}^2_{\Omega\times 
      Y^2} + \left(\min\{D\omega'\}-
      \frac{\bigC_D}{2\lambda}\right)\norm{\nabla_y \HE_{2,a}  -\nabla
      \HE_{2,b}}^2_{\Omega\times Y^2,t} \\ 
    &\leq \frac{\bigC_D\lambda}{2}\norm{(\HE_{1,a}-\HE_{1,b})\nabla
      \HE_{1,b}}^2_{\Omega,t} +
    \frac{\bigC_D\lambda}{2}\norm{(\HE_{2,a}-\HE_{2,b})\nabla_y
      \HE_{2,b}}^2_{\Omega\times Y^2,t}.
  \end{align*}
  Then we choose $\lambda$ large enough such that on the left-hand side
  all terms are positive, after which we can apply Gronwall's Lemma to
  obtain 
  \begin{align*}
    \norm{\HE_{1,a}&-\HE_{1,b}}^2_\Omega + \norm{\nabla \HE_{1,a} -
      \nabla \HE_{1,b}}^2_{\Omega,t} + \norm{\HE_{2,a} -
      \HE_{2,b}}^2_{\Omega\times Y^2} + \norm{\nabla_y \HE_{2,a} -\nabla
      \HE_{2,b}}^2_{\Omega\times Y^2,t} \leq 0.
  \end{align*}
  Consequently, $\HE_{1,a} = \HE_{1,b}$ and $\nabla \HE_{1,a} = \nabla
  \HE_{1,b}$ almost everywhere on $\Omega\times[0,\tend]$; similarly,
  $\HE_{2,a} = \HE_{2,b}$ and $\nabla \HE_{2,a} = \nabla \HE_{2,b}$
  almost everywhere on $\Omega\times Y^2\times[0,\tend]$.
\end{proof}

\bibliography{bibliothek}
\bibliographystyle{siam}

\end{document}